\documentclass[a4paper]{article}
\usepackage{hyperref}
\usepackage{cleveref}
\hypersetup{hypertexnames = false, bookmarksdepth = 2, bookmarksopen = true, colorlinks, linkcolor = black, citecolor = black, urlcolor = black, pdfstartview={XYZ null null 1}}

\usepackage{amsfonts}
\usepackage[fleqn, leqno]{amsmath}
\usepackage{amssymb}
\usepackage{amsthm}
\usepackage[maxbibnames=99, sortcites]{biblatex}
\usepackage{booktabs}
\usepackage{diagbox}
\usepackage{enumitem}
\usepackage{mathtools}
\usepackage{parskip}
\usepackage{thmtools}
\usepackage{tikz-cd}
\usepackage[colorinlistoftodos, textsize = footnotesize]{todonotes}
\usepackage{xparse}
\usepackage{xspace}

\usepackage[T1]{fontenc}
\usepackage{libertine}
\usepackage[libertine]{newtxmath}
\usepackage[scaled=0.83]{beramono}
\usepackage{eucal}
\usepackage{microtype}
\usepackage{gitinfo2}

\usepackage[all,ps,cmtip]{xy}
\usepackage{mathrsfs}
\newcounter{todocounter}
\DeclareDocumentCommand\addreference{g}{\stepcounter{todocounter}\todo[color = blue!30]{\thetodocounter. Add reference\IfNoValueF{#1}{: #1}}\xspace}
\DeclareDocumentCommand\checkthis{g}{\stepcounter{todocounter}\todo[color = red!50]{\thetodocounter. Check this\IfNoValueF{#1}{: #1}}\xspace}
\DeclareDocumentCommand\fixthis{g}{\stepcounter{todocounter}\todo[color = orange!50]{\thetodocounter. Fix this\IfNoValueF{#1}{: #1}}\xspace}
\DeclareDocumentCommand\expand{g}{\stepcounter{todocounter}\todo[color = green!50]{\thetodocounter. Expand\IfNoValueF{#1}{: #1}}\xspace}

\declaretheoremstyle[
  spaceabove = 3pt,
  spacebelow = 3pt,
]{lecture}
\theoremstyle{lecture}
\newtheorem{theorem}{Theorem}
\newtheorem{corollary}[theorem]{Corollary}

\newtheorem{definition}[theorem]{Definition}
\newtheorem{example}[theorem]{Example}

\newtheorem{lemma}[theorem]{Lemma}
\newtheorem{proposition}[theorem]{Proposition}
\newtheorem{remark}[theorem]{Remark}

\newtheorem{alphatheorem}{Theorem}

\crefname{alphatheorem}{Theorem}{Theorems}

\usetikzlibrary{arrows,arrows.meta}
\tikzcdset{arrow style=tikz, diagrams={>={Classical TikZ Rightarrow[]}}}

\makeatletter
\def\gitfootnote{\gdef\@thefnmark{}\@footnotetext}
\makeatother

\hypersetup{pagebackref}

\mathchardef\mhyphen="2D
\newcommand\dash{\nobreakdash-\hspace{0pt}}

\let\oldbigwedge\bigwedge
\renewcommand\bigwedge{\oldbigwedge\nolimits}

\newcommand\hilbtwo[1][X]{\ensuremath{{#1}^{[2]}}} % if we want to switch to Hilb^2X at some point, this is the place to do it
\newcommand\hilbn[2]{\ensuremath{#2^{[#1]}}} % if we want to switch to Hilb^nX at some point, this is the place to do it

\newcommand\bounded{\ensuremath{\mathrm{b}}}
\newcommand\LLL{\ensuremath{\mathbf{L}}}

\newcommand\RR{\ensuremath{\mathrm{R}}}
\newcommand\RRR{\ensuremath{\mathbf{R}}}
\newcommand\tangent{\ensuremath{\mathrm{T}}}

\DeclareMathOperator\Bl{Bl}
\DeclareMathOperator\ch{ch}
\DeclareMathOperator\coh{coh}
\DeclareMathOperator\Def{Def}
\DeclareMathOperator\derived{\mathbf{D}}
\DeclareMathOperator\Ext{Ext}
\DeclareMathOperator\HH{H}
\DeclareMathOperator\hh{h}
\DeclareMathOperator\HHHH{HH}
\DeclareMathOperator\hodgepoly{e}
\DeclareMathOperator\Hom{Hom}
\DeclareMathOperator\identity{id}
\DeclareMathOperator\moduli{M}
\DeclareMathOperator\Pic{Pic}
\DeclareMathOperator\sheafExt{\mathcal{E}xt}
\DeclareMathOperator\RsheafHom{\mathbf{R}\mathcal{H}om}
\DeclareMathOperator\sheafHom{\mathcal{H}om}
\DeclareMathOperator\Spec{Spec}
\DeclareMathOperator\Supp{Supp}
\DeclareMathOperator\Sym{Sym}
\DeclareMathOperator\td{td}

\title{Hilbert squares: derived categories and deformations}
\author{Pieter Belmans \and Lie Fu \and Theo Raedschelders}
\begin{document}
\maketitle

\begin{abstract}
  For a smooth projective variety $X$ with exceptional structure sheaf, and $\hilbtwo$ the Hilbert scheme of two points on $X$, we show that the Fourier--Mukai functor $\derived^\bounded(X) \to \derived^\bounded(\hilbtwo)$ induced by the universal ideal sheaf is fully faithful, provided the dimension of $X$ is at least 2. This fully faithfulness allows us to construct a spectral sequence relating the deformation theories of $X$ and $\hilbtwo$ and to show that it degenerates at the second page, giving a Hochschild--Kostant--Rosenberg-type filtration on the Hochschild cohomology of $X$. These results generalise known results for surfaces due to Krug--Sosna, Fantechi and Hitchin. Finally, as a by-product, we discover the following surprising phenomenon: for a smooth projective variety of dimension at least 3 with exceptional structure sheaf, it is rigid if and only if its Hilbert scheme of two points is rigid. This last fact contrasts drastically to the surface case: non-commutative deformations of a surface contribute to commutative deformations of its Hilbert square.
\end{abstract}

%\let\thefootnote\relax\footnotetext{MSC classes: 14F05, 14J10; keywords: derived categories, Hilbert schemes, Hochschild cohomology, deformation theory}

%\gitfootnote{commit: \texttt{\gitAbbrevHash}\hfil date: \texttt{\gitAuthorIsoDate}\hfil \texttt{\gitReferences}}

\tableofcontents

\section{Introduction}
\label{section:introduction}
Hilbert schemes of points are one of the most tractable and well-studied moduli spaces of sheaves. In this paper we discuss them from the point of view of derived categories and their deformations. For surfaces the situation is well-understood: the deformation theory is described by Fantechi and Hitchin in \cite{MR1354269,MR3024823}, and the derived categories of Hilbert schemes are studied by e.g.~Bridgeland--King--Reid, Haiman and Krug--Sosna \cite{MR3397451,MR1839919,MR1824990}.

For varieties of dimension~$\geq3$, a first obstacle is that the Hilbert scheme of points is not smooth, unless~$n=2,3$ \cite[\S0.2]{MR1616606}. In this paper we focus on the case of~$n=2$, i.e.~that of \emph{Hilbert squares}, and explain how results from the surface case generalise to higher dimensions.

More precisely, in \cite[theorem~9]{MR3024823}, Hitchin shows that for every smooth projective surface~$S$ with~$\HH^1(S,\mathcal{O}_S)=0$ there is a short exact sequence
\begin{equation}
  \label{equation:hitchin-short-exact}
  0 \to \HH^1(S,\tangent_S) \to \HH^1(\hilbn{n}{S},\tangent_{\hilbn{n}{S}}) \to \HH^0(S,\omega_S^{-1}) \to 0.
\end{equation}
This allows Hitchin to interpret the deformation theory of the Hilbert scheme~$\hilbn{n}{S}$ of~$n$~points on~$S$ in terms of the deformation theory and the Poisson geometry of~$S$.

Hitchin's construction is very geometric, and the observation that started this work is that \eqref{equation:hitchin-short-exact} can be recovered more abstractly using derived categories and Hochschild cohomology. For this we use a result due to Krug and Sosna: for a smooth projective surface~$S$ for which~$\mathcal{O}_S$ is exceptional (i.e.~$\HH^i(S,\mathcal{O}_S)=0$ for~$i\geq 1$) the Fourier--Mukai functor
\begin{equation}
  \Phi_{\mathcal{I}}\colon\derived^\bounded(S)\to\derived^\bounded(\hilbn{n}{S}),
\end{equation}
whose kernel~$\mathcal{I}$ is the ideal sheaf of the universal family~$Z\hookrightarrow S\times\hilbn{n}{S}$, is fully faithful \cite[theorem 1.2]{MR3397451}. It follows that there is an isomorphism
\begin{equation}
  \label{equation:hhiso}
  \HHHH^\bullet(S) \cong \Ext^\bullet_{S \times \hilbn{n}{S}}(\mathcal{I},\mathcal{I}).
\end{equation}
Moreover, the relative local-to-global Ext spectral sequence for the second projection~$S \times \hilbn{n}{S} \to \hilbn{n}{S}$ computing the right-hand side in \eqref{equation:hhiso} degenerates for dimension reasons, which readily allows one to recover \eqref{equation:hitchin-short-exact}, as we explain in \cref{subsection:degeneration-surfaces}.

To understand whether Hitchin's deformation-theoretic results for surfaces have an analogue in higher dimensions, it is therefore natural to first try and generalise the result of Krug--Sosna, and our first main result does exactly this.

\begin{alphatheorem}
  \label{theorem:fully-faithful}
  Let~$X$ be a smooth projective variety of dimension at least~2, such that~$\mathcal{O}_X$ is exceptional. Then
  \begin{equation}
    \Phi_{\mathcal{I}}\colon\derived^\bounded(X)\to\derived^\bounded(\hilbtwo)
  \end{equation}
  is fully faithful.
\end{alphatheorem}

It is noteworthy to mention that whereas the proof in \cite{MR3397451} relies on the Bridgeland--King--Reid--Haiman equivalence \cite{MR1824990,MR1839919}, our proof of \cref{theorem:fully-faithful} uses the explicit geometry of~$\hilbtwo$, which works in arbitrary dimension. In particular, it provides an alternative proof of the Krug--Sosna result in the case of~2~points. One should note that there are \emph{other} fully faithful functors from~$\derived^\bounded(X)$ into~$\derived^\bounded(\hilbtwo)$, see \cite{MR3811590}, which we discuss in \cref{remark:krug-ploog-sosna}. More recently it was shown that~$\Phi_{\mathcal{I}}$ is \emph{faithful}, for all~$X$ and~$n$ subject to a numerical condition which is satisfied when~$n=2$ \cite[theorem~1.3]{1808.05931v1}. We come back to this in \cref{remark:krug}. However, for our purposes it is important to both use the universal ideal sheaf~$\mathcal{I}$ for the kernel, and that the functor is \emph{full} and faithful.

One can then use a similar spectral sequence argument as in the case of surfaces to obtain generalisations of \eqref{equation:hitchin-short-exact}, and relate the deformation theories of~$X$ and~$\hilbtwo$. In higher dimensions, the degeneration does not come for free, and our second main theorem deals with this. The following statement is an abridged version of \cref{theorem:spectral-sequence}.

\begin{alphatheorem}
  \label{theorem:B}
  Let~$X$ be a smooth projective variety of dimension~$\geq 2$, such that~$\mathcal{O}_X$ is exceptional. The relative local-to-global Ext spectral sequence
  \begin{equation}
    \mathrm{E}_2^{i,j}=\HH^i(\hilbtwo,\sheafExt_{p_{\hilbtwo}}^j(\mathcal{I},\mathcal{I}))\Rightarrow\Ext_{X\times\hilbtwo}^{i+j}(\mathcal{I},\mathcal{I}).
  \end{equation}
  degenerates at the~$\mathrm{E}_2$\dash page. Moreover, the abutment can be identified with the Hochschild cohomology of~$X$, such that the induced filtration coincides with the filtration associated to the Hochschild--Kostant--Rosenberg decomposition, up to a degree shift.
\end{alphatheorem}

We then proceed to study the deformation theory of~$\hilbtwo$ for~$\dim X\geq 3$, where we obtain the following result, showing that the surface case is fundamentally different from the higher-dimensional situation.

\begin{alphatheorem}
  \label{theorem:C}
  Let~$X$ be a smooth projective variety of dimension at least~3, such that~$\mathcal{O}_X$ is exceptional. There exist isomorphisms
  \begin{equation}
    \HH^i(X,\tangent_X)\overset{\cong}{\to}\HH^i(\hilbtwo,\tangent_{\hilbtwo})
  \end{equation}
  for all~$i\geq 0$.
\end{alphatheorem}

This gives the following no-go result: if~$X$ is rigid, then so is~$\hilbtwo$. This is in stark contrast to the surface case: e.g.~if~$S=\mathbb{P}^2$ then~$S$ is rigid, but~$\hilbn{n}{S}$ is not \cite{MR2303228,MR3669875}, for~$n \geq 2$. These deformations are associated to noncommutative deformations of~$S$. We come back to this in \cref{remark:noncommutative-deformations}.

For Hilbert cubes we expect similar results to hold, but the geometric description of~$\hilbn{3}{X}$ is more involved, making the analysis harder.

The paper is structured as follows. In \cref{section:preliminaries} we recall some preliminaries regarding the geometry of Hilbert squares and Hochschild cohomology for smooth projective varieties. The proof of \cref{theorem:fully-faithful} is given in \cref{section:fully-faithful}.

In \cref{section:relative-ext} we prove \cref{theorem:B}, in the form of \cref{theorem:spectral-sequence}. The proof in arbitrary dimension is given in \cref{subsection:degeneration}, whilst the special case of surfaces is discussed with more details in \cref{subsection:degeneration-surfaces}.

Finally, in \cref{section:deformation-hilbert-squares} we discuss the deformation theory of Hilbert squares. In \cref{subsection:dimension-2} we explain how the methods in this paper recover known results for the deformation theory of Hilbert squares of surfaces, and speculate on a noncommutative generalisation of our results. In \cref{subsection:dim-3-interpretation} we prove \cref{theorem:C}.

\paragraph{Acknowledgements}
We would like to thank Daniel Huybrechts, Zhi Jiang, Joe Karmazyn and Alexander Perry for interesting discussions.

The first author is supported by the Max Planck Institute for Mathematics in Bonn. The second author is supported by ECOVA (ANR-15-CE40-0002), HodgeFun (ANR-16-CE40-0011), LABEX MILYON (ANR-10-LABX-0070) of Universit\'e de Lyon and \emph{Projet Inter-Laboratoire} 2017 and 2018 by F\'ed\'eration de Recherche en Math\'ematiques Rh\^one-Alpes/Auvergne CNRS 3490. The third author is supported by an EPSRC postdoctoral fellowship EP/R005214/1.

\section{Preliminaries}
\label{section:preliminaries}
In this section we discuss some preliminaries, and introduce the notation that is used throughout.

\subsection{Hilbert squares}
\label{subsection:hilbert-squares}
The Hilbert scheme of length-2 subschemes (or \emph{Hilbert square}) of~$X$ is denoted ~$\hilbtwo$, and is a smooth projective variety of dimension~$2d$. We will often make use of the diagram
\begin{equation}
  \label{equation:blowup-diagram}
  \begin{tikzcd}
    E \arrow[dr, phantom, "\square"] \arrow[r, hook, "j"] \arrow[d, "p"] & \Bl_\Delta(X \times X) \arrow[r, "q"] \ar[d, "\pi"] & \hilbtwo \\
    X \arrow[r, hook, "\Delta"] & X \times X
  \end{tikzcd}
\end{equation}
where the left square is the blowup square of~$X \times X$ along its diagonal. The natural involution on~$X \times X$ lifts to the blowup and the quotient is canonically isomorphic to the Hilbert square of~$X$, denoted by~$\hilbtwo$. The exceptional divisors in~$\hilbtwo$ and~$\Bl_\Delta(X \times X)$ can be compared using the following diagram
\begin{equation}
  \label{equation:divisor-diagram}
  \begin{tikzcd}
    E \arrow[r, hook] \arrow[dr, swap, "\cong"] & 2E \arrow[dr, phantom, "\square"] \arrow[d] \arrow[r, hook] & \Bl_\Delta(X \times X) \arrow[d, "q"] \\
    & D \arrow[r, hook] & \hilbtwo
  \end{tikzcd}
\end{equation}
where~$2E$ is the second infinitesimal neighborhood of~$E$ in~$\Bl_\Delta(X \times X)$, and~$D$ is the (isomorphic) image of~$E$ by~$q$ which is the locus parametrizing non-reduced length-2 subschemes of~$X$. The divisors on~$\hilbtwo$ are well understood up to rational equivalence:
\begin{equation}
  \Pic(\hilbtwo) \cong \Pic(X) \oplus \mathbb{Z}[\delta]
\end{equation}
where~$\delta$ is a divisor on~$\hilbtwo$ such that~$\mathcal{O}_{\hilbtwo}(2\delta)=\mathcal{O}_{\hilbtwo}(D)$. Using this, we can show the following lemma, which we will use throughout.

\begin{lemma}
  \label{lemma:HodgeNumberHilb}
  If a smooth projective variety~$X$ has an exceptional structure sheaf, then so does the Hilbert square~$\hilbtwo$.
\end{lemma}

\begin{proof}
  The K\"unneth formula shows that~$X\times X$ has an exceptional structure sheaf, hence so does the blowup~$\Bl_\Delta(X\times X)$ by the birational invariance of the cohomology of the structure sheaf. Since
  \begin{align}
    \HH^i(\Bl_\Delta(X \times X), \mathcal{O}_{\Bl_\Delta(X \times X)})
    &\cong \HH^i(\hilbtwo, q_*\mathcal{O}_{\Bl_\Delta(X \times X)})\\
    &\cong \HH^i(\hilbtwo, \mathcal{O}_{\hilbtwo})\oplus \HH^i(\hilbtwo, \mathcal{O}_{\hilbtwo}(-\delta)),
  \end{align}
  we see that~$\HH^i(\hilbtwo,\mathcal{O}_{\hilbtwo})$ is a direct summand of~$\HH^i(\Bl_\Delta(X \times X),\mathcal{O}_{\Bl_\Delta(X \times X)})$, and hence vanishes for any~$i\geq 1$.
\end{proof}

\begin{remark}
  More generally the description of~$\hilbtwo$ as the blowup of the quotient allows us to readily describe the Hodge polynomial\footnote{For surfaces and arbitrary~$n\geq2$, the Hodge polynomial is described by G\"ottsche--Soergel via a generating series involving all values of~$n$ \cite{MR1219901}.}. Recall that for a smooth projective variety~$Y$ its polynomial is defined as
  \begin{equation}
    \hodgepoly(Y)(x,y)\coloneqq\sum_{p,q=0}^{\dim Y}(-1)^{p+q}\hh^{p,q}(Y)x^py^q,
  \end{equation}
  where~$\hh^{p,q}(Y)\coloneqq\dim_k\HH^q(Y,\Omega_Y^p)$. One can then show (e.g.~as in \cite[lemma~2.6]{MR2506383}) that the Hodge polynomial of~$\hilbtwo$ is
  \begin{equation}
    \hodgepoly(\hilbtwo)(x,y)=\frac{1}{2}\left( \hodgepoly(X)(x,y)^2 + \hodgepoly(X)(x^2,y^2) \right)+\sum_{i=0}^{d-2}\hodgepoly(X)(x,y)x^{i+1}y^{i+1}.
  \end{equation}
  This also proves \cref{lemma:HodgeNumberHilb}, and is useful when considering the semiorthogonal decomposition of \cref{remark:krug-ploog-sosna} from the point of view of additive invariants.
  % lemma 2.6 of arXiv:0701642 and theorem 11 of 1301.0478
  % see code/hodge.sage for an implementation
\end{remark}

Finally, because the Hilbert scheme is a fine moduli space, there is a \emph{universal closed subscheme}
\begin{equation}
  \begin{tikzcd}
    Z\coloneqq\{(x,\xi) \mid x \in \Supp(\xi)\} \arrow[r, hook, "i"] &  X \times \hilbtwo
  \end{tikzcd}
\end{equation}
which is smooth, and there is an isomorphism
\begin{equation}
  Z \cong \Bl_\Delta(X \times X).
\end{equation}
There is a corresponding universal short exact sequence on~$X\times\hilbtwo$
\begin{equation}
  \label{equation:universal}
  0 \to \mathcal{I} \to \mathcal{O}_{X \times \hilbtwo} \to \mathcal{O}_Z \to 0,
\end{equation}
with~$\mathcal{I}$ the \emph{universal ideal sheaf}.

%\begin{equation}
%  \begin{tikzcd}
%    & X\times\hilbtwo \arrow[ld, "p_1"] \arrow[rd, "p_2"] \\
%    X & & \hilbtwo
%  \end{tikzcd}
%\end{equation}

\subsection{Hochschild cohomology}
\label{subsection:hochschild-cohomology}
We will also use Hochschild cohomology as an invariant of the derived category, let us briefly recall its definition and main properties. For more information one is referred to \cite{0904.4330v1} or \cite[\S5.2]{MR2244106}.

\begin{definition}
  Let~$X$ be a smooth projective variety. Then its \emph{Hochschild cohomology} is defined as the self-extensions of the identity functor, i.e.
  \begin{equation}
    \HHHH^n(X)\coloneqq\Ext_{X\times X}^n(\Delta_*\mathcal{O}_X,\Delta_*\mathcal{O}_X)
  \end{equation}
\end{definition}

The Hochschild--Kostant--Rosenberg decomposition then expresses Hochschild cohomology in terms of sheaf cohomology of polyvector fields. For a proof in this context we refer to \cite[theorem~4.1]{MR2472137}.

\begin{theorem}[Hochschild--Kostant--Rosenberg]
  Let~$X$ be a smooth projective variety. Then there exists an isomorphism
  \begin{equation}
    \HHHH^n(X)\cong\bigoplus_{p+q=n}\HH^p\left( X,\bigwedge^q\tangent_X \right).
  \end{equation}
\end{theorem}

The second Hochschild cohomology group of~$X$ describes the deformation theory of the abelian category~$\coh X$ as explained in \cite{MR2238922}. Under the Hochschild--Kostant--Rosenberg decomposition we can give an interpretation to the three components which arise in the decomposition, as explained in \cite{MR2477894}. Namely we have that
\begin{equation}
  \label{equation:hkr-hh-2}
  \HHHH^2(X)\cong\underset{\text{gerby}}{\HH^2(X,\mathcal{O}_X)}\oplus\underset{\text{commutative}}{\HH^1(X,\tangent_X)}\oplus\underset{\text{noncommutative}}{\HH^0\left( X,\bigwedge^2\tangent_X \right)},
\end{equation}
Then
\begin{enumerate}
  \item $\HH^2(X,\mathcal{O}_X)$ describes the gerby deformations;
  \item $\HH^1(X,\tangent_X)$ is part of the usual Kodaira--Spencer deformation-obstruction calculus, describing infinitesimal deformations of~$X$ as a variety;
  \item $\HH^0\left( X,\bigwedge^2\tangent_X \right)$ are (pre-)Poisson structures, and describe noncommutative infinitesimal deformations of the structure sheaf~$\mathcal{O}_X$ as a sheaf of algebras.
\end{enumerate}
In this article we will ignore the component~$\HH^2(X,\mathcal{O}_X)$ as it is zero for all varieties considered here, by \cref{lemma:HodgeNumberHilb} and the assumption that~$\mathcal{O}_X$ is exceptional.

\subsection{Summary of notation}
We will fix the following notation. Because we have that~$\hilbtwo\cong\Bl_\Delta(X\times X)/(\mathbb{Z}/2\mathbb{Z})$, we can consider the blowup diagram together with the quotient morphism. Moreover the universal closed subscheme can be used to extend this diagram as follows.
\begin{equation}
  \label{equation:notation}
  \begin{tikzcd}
    & & \hilbtwo \\
    E\cong\mathbb{P}(\tangent_X) \arrow[dr, phantom, "\square"]  \arrow[r, hook, "j"] \arrow[d, "p"] & \Bl_\Delta(X \times X)\cong Z \arrow[r, hook, "i"] \arrow[ru, "q"] \arrow[d, "\pi"] & X\times\hilbtwo \arrow[d, two heads, "p_X"] \arrow[u, two heads, swap, "p_{\hilbtwo}"] \\
    X \arrow[r, hook, "\Delta"] & X \times X \arrow[r, "p_1"] & X
  \end{tikzcd}
\end{equation}
Here~$E$ is the exceptional divisor, which is the projective bundle over~$X$ given by the normal bundle of the diagonal embedding, i.e.~the tangent bundle. The morphism~$p_1$ is the projection on the first factor. We will denote the different projections in products involving~$X$ and~$\hilbtwo$ as
\begin{equation}
  \label{equation:projection-diagram-notation}
  \begin{tikzcd}
    & X\times\hilbtwo \arrow[ld, swap, "p_X"] \arrow[rd, "p_{\hilbtwo}"] \\
    X & & \hilbtwo
  \end{tikzcd}
\end{equation}
and
\begin{equation}
  \label{equation:projection-diagram-triple-notation}
  \begin{tikzcd}
    & X\times\hilbtwo\times X \arrow[ld, swap, "p_{1,2}"] \arrow[d, "p_{1,3}"] \arrow[rd, "p_{2,3}"] \\
    X\times\hilbtwo & X\times X & \hilbtwo\times X
  \end{tikzcd}.
\end{equation}

\paragraph{Conventions}
Throughout,~$k$ will denote an algebraically closed field of characteristic~$0$, and~$X$ denotes a smooth projective variety of dimension~$d \geq 2$ with an exceptional structure sheaf, i.e.~$\HH^i(X,\mathcal{O}_X)=0$ for all~$i\geq 1$, unless explicitly mentioned otherwise.

\section{Fully faithfulness}
\label{section:fully-faithful}
In this section we prove \cref{theorem:fully-faithful}. The global structure of the proof is analogous to the proof of \cite[theorem~1.2]{MR3397451}, but the actual steps are completely different, as there is no derived McKay correspondence that can be used to compute the different Fourier--Mukai kernels.

We will denote the Fourier--Mukai functors~$\derived^\bounded(X) \to \derived^\bounded(\hilbtwo)$ associated to the sheaves in \eqref{equation:universal}
\begin{equation}
	\begin{aligned}
		F&\coloneqq\Phi_{\mathcal{I}}, \\
		F'&\coloneqq\Phi_{\mathcal{O}_{X \times \hilbtwo}}, \\
		F''&\coloneqq\Phi_{\mathcal{O}_Z},
	\end{aligned}
\end{equation}
and similarly for their respective right adjoints:
\begin{equation}
\label{equation:right-adjoints}
	\begin{aligned}
		R&\coloneqq\Phi_{\mathcal{I}^\vee \otimes p_X^*\omega_X[d]} \\
		R'&\coloneqq\Phi_{p_X^*\omega_X[d]} \\
		R''&\coloneqq\Phi_{\mathcal{O}_Z^\vee \otimes p_X^*\omega_X[d]}
	\end{aligned}
\end{equation}
where~$(-)^\vee\coloneqq\RsheafHom(-, \mathcal{O}_{X\times \hilbtwo})$ is the derived dual. This notation for the functors mimicks that of \cite{MR3397451}.

To show fully faithfulness, it suffices to show that the unit
\begin{equation}
\identity_{\derived^\bounded(X)} \to R\circ F
\end{equation}
of the adjunction is an equivalence. The composition of Fourier--Mukai functors is again a Fourier--Mukai functor, whose kernel is given by the convolution product of the respective kernels. In this case the kernel is~$(\mathcal{I}^\vee\otimes p_X^*\omega_X[d])*\mathcal{I}$.

Hence, by applying~$-\ast\mathcal{I}$ to the distinguished triangle
\begin{equation}
  \mathcal{O}_Z^\vee \otimes p_X^*\omega_X[d] \to p_X^*\omega_X[d] \to \mathcal{I}^\vee \otimes p_X^*\omega_X[d] \to
\end{equation}
coming from the right adjoints \eqref{equation:right-adjoints} and the universal short exact sequence \eqref{equation:universal}, it suffices to show that
\begin{equation}
  \label{equation:comp-adjoints}
	\begin{aligned}
		\left( \mathcal{O}_Z^\vee \otimes p_X^*\omega_X[d] \right) \ast\mathcal{I} &\cong \mathcal{O}_{\Delta}[-1],\\
		\left( p_X^*\omega_X[d] \right) \ast \mathcal{I} &\cong 0,
  \end{aligned}
\end{equation}
since the left-hand sides of \eqref{equation:comp-adjoints} are exactly the Fourier--Mukai kernels of~$R''\circ F$ and~$R'\circ F$. This is done in \cref{proposition:iso-1,proposition:iso-2}. The hardest kernel to compute is the one for~$R''\circ F$, and we do this by first computing the kernels of~$R''\circ F'$ and of~$R''\circ F''$ in \cref{corollary:R''F',corollary:R''F''}.

% this already appears nearly verbatim in the introduction
%\begin{remark}
%  This is the same proof strategy that was used by Krug and Sosna in \cite[theorem~1.2]{MR3397451}. Their proof shows that for any smooth projective surface~$S$ with exceptional structure sheaf, the Fourier--Mukai functor
%  \begin{equation}
%    \Phi_{\mathcal{I}}\colon\derived^\bounded(S) \to \derived^\bounded(\hilbn{n}{S})
%  \end{equation}
%  is fully faithful, for any~$n \geq 2$. The computation makes use of the Bridgeland--King--Reid--Haiman equivalence. Our proof works in arbitrary dimension but only for~$n=2$, and instead makes use of the explicit geometry of the Hilbert square. Therefore it also gives a new and self-contained proof of the fully faithfulness for the Hilbert square of a surface.
%\end{remark}

The following lemma computes the convolution of a Fourier--Mukai kernel on~$X\times\hilbtwo$ with the Fourier--Mukai kernel for~$R''$.
\begin{lemma}
  For every object $\mathcal{E}\in \derived^{\bounded}(X\times \hilbtwo)$, there is an isomorphism
  \begin{equation}
    \label{equation:convolution}
    \left( (i_*\mathcal{O}_Z)^\vee\otimes p_X^*\omega_X[d] \right) \ast \mathcal{E} \cong \RRR({\identity_X}\times (p_1\circ \pi))_*\left( ({\identity_X}\times q)^*\mathcal{E}\otimes p_Z^*\mathcal{O}_Z(E) \right).
  \end{equation}
\end{lemma}

\begin{proof}
  Using Grothendieck duality we compute the derived dual as
  \begin{equation}
    \label{equation:derived-dual-OZ}
    \begin{aligned}
      (i_*\mathcal{O}_Z)^\vee&\cong\RsheafHom\left(i_*\mathcal{O}_Z, \mathcal{O}_{X\times \hilbtwo}\right) \\
      &\cong i_*\circ i^!(\mathcal{O}_{X\times \hilbtwo}) \\
      &\cong i_*(\omega_i[\dim i]) \\
      &\cong i_*\omega_i[-d],
    \end{aligned}
  \end{equation}
  where~$\omega_i\coloneqq\omega_Z\otimes i^*\omega_{X\times \hilbtwo}^{-1}$ is the relative canonical bundle of $i$ and
  \begin{equation}
    \dim i=\dim Z-\dim (X\times \hilbtwo)=-d
  \end{equation}
  is its relative dimension. A straightforward computation using functoriality and the projection formula gives that for any object $\mathcal{E}\in \derived^{\bounded}(X\times \hilbtwo)$,
  \begin{equation}
    \begin{aligned}
      &\left( (i_*\mathcal{O}_Z)^\vee\otimes p_X^*\omega_X[d] \right)\ast \mathcal{E} \\
      &\qquad\coloneqq \RRR p_{1,3,*}\left(p_{1,2}^*\mathcal{E} \otimes^\LLL p_{2,3}^*\left( (i_*\mathcal{O}_Z)^\vee\otimes p_X^*\omega_X\right) \right)[d]\\
      &\qquad\cong\RRR p_{1,3,*}\left(p_{1,2}^*\mathcal{E} \otimes^\LLL p_{2,3}^*(i_*\mathcal{O}_Z)^\vee\otimes p_{3}^*\omega_X \right)[d]\\
      &\qquad\cong\RRR p_{1,3,*}\left(p_{1,2}^*\mathcal{E} \otimes^\LLL p_{2,3}^*(i_*\mathcal{O}_Z)^\vee\right)\otimes p_2^*\omega_X [d]\\
      &\qquad\cong\RRR p_{1,3,*}\left(p_{1,2}^*\mathcal{E} \otimes^\LLL p_{2,3}^*(i_*\omega_i)\right)\otimes p_2^*\omega_X.
    \end{aligned}
  \end{equation}
  With the help of the following diagram where the left square is cartesian
  \begin{equation}
    \label{equation:commutative-diagram-triple-projections}
    \begin{tikzcd}
      X\times \Bl_\Delta(X \times X) \arrow[dr, phantom, "\square"] \arrow[r, hook, "{\identity_X}\times i"] \arrow[d, "p_Z"] & X\times \hilbtwo\times X \arrow[d, "p_{2,3}"] \arrow[r, "p_{1,3}"] \arrow[dr, "p_{1,2}"] & X\times X \\
      Z\cong\Bl_\Delta(X \times X) \arrow[r, hook, "i"] & \hilbtwo\times X & X\times \hilbtwo
    \end{tikzcd}
  \end{equation}
  we can continue the computation using base change and the projection formula
  \begin{equation}
    \begin{aligned}
      %\phantom{\left( (i_*\mathcal{O}_Z)^\vee\otimes p_X^*\omega_X[d] \right)\circ \mathcal{E}}
      &\qquad\cong\RRR p_{1,3,*}\left( p_{1,2}^*\mathcal{E} \otimes^\LLL ({\identity_X}\times i)_*\circ p_Z^*(\omega_i) \right)\otimes p_2^*\omega_X \\
      &\qquad\cong\RRR p_{1,3,*}({\identity_X}\times i)_*\left(({\identity_X}\times i)^*\circ p_{1,2}^*(\mathcal{E})\otimes p_Z^*\omega_i\right)\otimes p_2^*\omega_X \\
      &\qquad\cong\RRR({\identity_X}\times (p_1\circ \pi))_*\left(({\identity_X}\times q)^*\mathcal{E}\otimes p_Z^*\omega_i\right)\otimes p_2^*\omega_X.
    \end{aligned}
  \end{equation}
  Note that the last equality follows because
  \begin{equation}
    \begin{aligned}
      p_{1,3} \circ ({\identity_X} \times i)&={\identity_X} \times (p_1 \circ \pi), \\
      p_{1,2} \circ ({\identity_X} \times i)&={\identity_X} \times q,
    \end{aligned}
  \end{equation}
  by combining the commutative diagrams \eqref{equation:notation} and \eqref{equation:commutative-diagram-triple-projections}.

  Now the relative canonical bundle can be expressed as follows
  \begin{equation}
    \omega_i\coloneqq\omega_Z\otimes i^*\omega_{X\times \hilbtwo}^{-1}\cong\omega_Z\otimes q^*\omega_{\hilbtwo}^{-1}\otimes \pi^*\circ p_1^*(\omega_X^{-1})\cong\mathcal{O}_Z(E) \otimes \pi^*\circ p_1^*(\omega_X^{-1}),
  \end{equation}
  where we used that~$\omega_Z\otimes q^*\omega_{\hilbtwo}^{-1}$ can be computed by taking determinants of the exact sequence
  \begin{equation}
    0 \to \tangent_Z \to q^*\tangent_{\hilbtwo} \to \mathcal{O}_E(2E) \to 0,
  \end{equation}
  and using that~$\det(\mathcal{O}_E(2E))\cong\mathcal{O}_Z(2E) \otimes \mathcal{O}_Z(-E)\cong\mathcal{O}_Z(E)$. Hence
  \begin{equation}
    \label{equation:R''*}
    \left( (i_*\mathcal{O}_Z)^\vee\otimes p_X^*\omega_X[d] \right) \ast \mathcal{E}\cong \RRR({\identity_X}\times (p_1\circ \pi))_*\left( ({\identity_X}\times q)^*\mathcal{E}\otimes p_Z^*\mathcal{O}_Z(E) \right),
  \end{equation}
  where we used the cartesian diagram
  \begin{equation}
    \label{equation:cartesian}
    \begin{tikzcd}[row sep=large, column sep=large]
      X\times \Bl_\Delta(X\times X) \arrow[d, "p_Z"] \ar[r, "{\identity_X}\times (p_1\circ \pi)"] \arrow[dr, phantom, "\square"] &X\times X \arrow[d, "p_2"] \\
      Z\cong\Bl_\Delta(X \times X) \arrow[r, "p_1\circ \pi"] & X
    \end{tikzcd}
  \end{equation}
  and the projection formula.
\end{proof}

We will apply this with~$\mathcal{E}\coloneqq\mathcal{O}_{X\times\hilbtwo}$ (resp.~$i_*\mathcal{O}_Z)$) to compute~$R''\circ F'$ (resp.~$R''\circ F''$) in the following two corollaries.

\begin{corollary}
  \label{corollary:R''F'}
  The Fourier--Mukai kernel of~$R''\circ F'$ is~$\mathcal{O}_{X \times X}$.
\end{corollary}

\begin{proof}
  Taking~$\mathcal{E}\coloneqq\mathcal{O}_{X\times \hilbtwo}$ in \eqref{equation:convolution}, the kernel of $R''\circ F'$ is
  \begin{equation}
    \label{equation:R''F'}
    \begin{aligned}
      \left( (i_*\mathcal{O}_Z)^\vee\otimes p_X^*\omega_X[d] \right) \ast \mathcal{O}_{X\times \hilbtwo} &\cong \RRR({\identity_X}\times (p_1\circ \pi))_*\left(p_Z^*\mathcal{O}_Z(E)\right)\\
      &\cong p_2^*\left(\RRR(p_1\circ \pi)_*\mathcal{O}_Z(E)\right)\\
      &\cong \mathcal{O}_{X\times X},
    \end{aligned}
  \end{equation}
  where the second isomorphism uses base change for the cartesian diagram \eqref{equation:cartesian}. The last isomorphism follows by applying~$\RRR\pi_*$ to
  \begin{equation}
    0 \to \mathcal{O}_Z \to \mathcal{O}_Z(E) \to \mathcal{O}_E(E) \to 0,
  \end{equation}
  and using that~$\RRR\pi_*\mathcal{O}_Z \cong \mathcal{O}_{X \times X}$, together with~$\RRR\pi_*\mathcal{O}_E(E)\cong\Delta_{*}\circ\RRR p_*\mathcal{O}_p(-1)=0$ as $p$ is a $\mathbb{P}^{d-1}$-bundle.
\end{proof}

\begin{corollary}
  \label{corollary:R''F''}
  The Fourier--Mukai kernel of~$R''\circ F''$ is~$\mathcal{O}_{\Delta_X} \oplus \mathcal{O}_{X \times X}$.
\end{corollary}

\begin{proof}
  We define~$Z_1\cup Z_2$ using the cartesian diagram
  \begin{equation}
    \label{equation:glue}
    \begin{tikzcd}
      Z_1\cup Z_2 \arrow[r, hook, "\iota"] \arrow[d] \arrow[dr, phantom, "\square"] & X\times Z \arrow[d, "\identity_X\times q"] \\
      Z \arrow[r, hook, "i"] & X\times \hilbtwo.
    \end{tikzcd}
  \end{equation}
  Moreover there exist isomorphisms~$\phi\colon Z_1 \cup Z_2 \cong Z \cup_E Z$ exhibiting~$Z_1\cup Z_2$ as the gluing along the exceptional divisor~$E$, transversally, of its two irreducible components, both of which are isomorphic to~$Z$. Moreover, we can assume that the left vertical morphism in \eqref{equation:glue} is the identity on each component by changing~$\phi$ if necessary. In this case we have
  \begin{equation}
    \iota_j\coloneqq\iota|_{Z_j}=(p_j \circ \pi,\identity)\colon Z_j \cong Z \to X \times Z.
  \end{equation}
  If we now take~$\mathcal{E}$ to be~$i_*\mathcal{O}_Z$ in \eqref{equation:convolution}, then the kernel of~$R''\circ F''$ is
  \begin{equation}
    \begin{aligned}
      \left( (i_*\mathcal{O}_Z)^\vee\otimes p_X^*\omega_X[d] \right)\ast i_*\mathcal{O}_Z &\cong \RRR({\identity_X}\times (p_1\circ \pi))_*\left(({\identity_X}\times q)^*\circ i_*\mathcal{O}_Z\otimes p_Z^*\mathcal{O}_Z(E)\right)\\
    &\cong \RRR({\identity_X}\times (p_1\circ \pi))_*\left(\iota_*\mathcal{O}_{Z_1\cup Z_2}\otimes p_Z^*\mathcal{O}_Z(E)\right),
    \end{aligned}
  \end{equation}
  where the second isomorphism follows from \eqref{equation:glue} by base change.

  The short exact sequence
  \begin{equation}
    0\to \mathcal{O}_{Z_1}(-E)\to \mathcal{O}_{Z_1\cup Z_2}\to \mathcal{O}_{Z_2}\to 0
  \end{equation}
  gives a distinguished triangle
  \begin{equation}
    \RRR({\identity_X}\times (p_1\circ \pi))_*\left(\iota_{1,*}\mathcal{O}_{Z_1}\right) \to \RRR({\identity_X}\times (p_1\circ \pi))_*\left(\iota_*\mathcal{O}_{Z_1\cup Z_2}\otimes p_Z^*\mathcal{O}_Z(E)\right) \to \RRR({\identity_X}\times (p_1\circ \pi))_*\left(\iota_{2,*}\mathcal{O}_{Z_2}(E)\right) \xrightarrow{+1}
  \end{equation}
  Since
  \begin{itemize}
    \item the composition~$Z_1\xrightarrow{\iota_1} X\times Z \xrightarrow{{\identity_X}\times (p_1\circ \pi)} X\times X$  is~$\Delta\circ p_1\circ \pi$;
    \item the composition~$Z_2\xrightarrow{\iota_2} X\times Z\xrightarrow{{\identity_X}\times (p_1\circ \pi)} X\times X$ is~$s\circ \pi$,
  \end{itemize}
  with~$s$ being the natural involution on~$X\times X$, the previous distinguished triangle is nothing but
  \begin{equation}
    \RRR(\Delta\circ p_1\circ \pi)_*\mathcal{O}_{Z_1}\to\left( (i_*\mathcal{O}_Z)^\vee\otimes p_X^*\omega_X[d] \right)\ast i_*\mathcal{O}_Z\to \RRR(s\circ \pi)_*\mathcal{O}_{Z_2}(E) \xrightarrow{+1}
  \end{equation}
  which is the following by the vanishing of~$\HH^i(X, \mathcal{O}_X)$ for~$i\geq 1$,
  \begin{equation}
    \mathcal{O}_{\Delta_X}\to\left( (i_*\mathcal{O}_Z)^\vee\otimes p_X^*\omega_X[d] \right)\ast i_*\mathcal{O}_Z \to \mathcal{O}_{X\times X} \xrightarrow{+1}.
  \end{equation}
  Now since the extension class of this distinguished triangle lives in
  \begin{equation}
    \Hom_{\derived^\bounded(X\times X)}(\mathcal{O}_{X\times X}, \mathcal{O}_{\Delta_X}[1])\cong \HH^1(X, \mathcal{O}_X)=0,
  \end{equation}
  the kernel of~$R''\circ F''$ is
  \begin{equation}
    \label{equation:R''F''}
    (i_*\mathcal{O}_Z)^\vee\otimes p_X^*\omega_X[d] \ast i_*\mathcal{O}_Z \cong \mathcal{O}_{\Delta_X}\oplus \mathcal{O}_{X\times X}
  \end{equation}
\end{proof}

These corollaries allow us to obtain the following proposition.

\begin{proposition}
  \label{proposition:iso-1}
  There is an isomorphism of Fourier--Mukai kernels
  \begin{equation}
    \left( \mathcal{O}_Z^\vee \otimes p_X^*\omega_X[d] \right)\ast\mathcal{I} \cong \mathcal{O}_{\Delta}[-1],
  \end{equation}
  and hence~$R''\circ F \cong [-1]$.
\end{proposition}

\begin{proof}
  In the distinguished triangle~$R''\circ F\to R''\circ F'\to R''\circ F''\xrightarrow{+1}$ (or rather that of their kernels), it is straightforward to check that the second morphism is induced by~$(0, \identity)\colon \mathcal{O}_{X\times X} \hookrightarrow \mathcal{O}_{\Delta_X}\oplus \mathcal{O}_{X\times X}$, via the isomorphisms \eqref{equation:R''F'} and \eqref{equation:R''F''}. Therefore, the Fourier--Mukai kernel of~$R''\circ F$ is~$\mathcal{O}_{\Delta_X}[-1]$.
\end{proof}

\begin{proposition}
  \label{proposition:iso-2}
  There is an isomorphism of Fourier--Mukai kernels
  \begin{equation}
    p_X^*\omega_X[d] \ast \mathcal{I} \cong 0,
  \end{equation}
  and hence~$R'\circ F \cong 0$.
\end{proposition}

\begin{proof}
  The following isomorphisms are obtained by functoriality and the projection formula
  \begin{equation}
    \begin{aligned}
      p_X^*\omega_X[d]\ast\mathcal{I}
      &\coloneqq\RRR p_{1,3,*}\left(p_{1,2}^*\mathcal{I}\otimes^\LLL p_{2,3}^*\circ p_X^*(\omega_X)\right)[d]\\
      &\cong\RRR p_{1,3,*}\left(p_{1,2}^*\mathcal{I}\otimes^\LLL p_{3}^*(\omega_X)\right)[d]\\
      &\cong\RRR p_{1,3,*}\left(p_{1,2}^*\mathcal{I}\otimes^\LLL (p_{1,3}^*\circ p_2^*)(\omega_X)\right)[d]\\
      &\cong\RRR p_{1,3,*}\left(p_{1,2}^*\mathcal{I}\right)\otimes^\LLL p_2^*\omega_X[d].
    \end{aligned}
  \end{equation}
  Thanks to the cartesian diagram
  \begin{equation}
    \begin{tikzcd}
      X\times\hilbtwo\times X \arrow[dr, phantom, "\square"] \arrow[r, "p_{1,2}"] \arrow[d, "p_{1,3}"] & X\times\hilbtwo \arrow[d, "p_X"] \\
      X\times X \arrow[r, "p_1"] & X
    \end{tikzcd}
  \end{equation}
  we have by base change, % tag 08IB
  \begin{equation}
    \RRR p_{1,3,*}\left( p_{1,2}^*\mathcal{I} \right)\cong p_1^*\left( \RRR p_{X,*}(\mathcal{I}) \right).
  \end{equation}

  Therefore, it is enough to show that~$\RRR p_{X,*}\mathcal{I}=0$. To this end, apply~$\RRR p_{X,*}$ to the universal short exact sequence \eqref{equation:universal}, so we obtain a distinguished triangle
  \begin{equation}
    \label{equation:RpSES}
    \RRR p_{X,*}\mathcal{I}\to \RRR p_{X,*}\mathcal{O}_{X\times \hilbtwo}\to \RRR p_{X,*}\circ i_*(\mathcal{O}_Z)\xrightarrow{+1}.
  \end{equation}
  Now by \cref{lemma:HodgeNumberHilb} we have the isomorphism
  \begin{equation}\label{equation:RpO}
    \RRR p_{X,*}\mathcal{O}_{X\times \hilbtwo}\cong\mathcal{O}_X.
  \end{equation}
  and by commutativity of the second square in \eqref{equation:notation} we obtain that
  \begin{equation}
    \label{equation:RpOZ}
    \RRR p_{X,*}\circ i_*(\mathcal{O}_Z)\cong \RRR p_{1,*}\circ\RRR\pi_*(\mathcal{O}_Z)\cong \RRR p_{1,*}\mathcal{O}_{X\times X}\cong\mathcal{O}_X,
  \end{equation}
  where the last equality uses the assumption on the vanishing of~$\HH^i(X,\mathcal{O}_X)$ for $i\geq 1$.

  Finally, one checks that the second morphism in \eqref{equation:RpSES} is, via the above isomorphisms \eqref{equation:RpO} and \eqref{equation:RpOZ}, the identity of $\mathcal{O}_X$. Thus~$\RRR p_{X,*}\mathcal{I}\cong 0$ as desired.
\end{proof}

We can now combine all these ingredients and give the proof of \cref{theorem:fully-faithful}.
\begin{proof}[Proof of \cref{theorem:fully-faithful}]
  By combining the triangle
  \begin{equation}
    \left( \mathcal{O}_Z^\vee \otimes p_X^*\omega_X[d]\right) \ast \mathcal{I} \to \left( p_X^*\omega_X[d] \right)\ast \mathcal{I} \to \left( \mathcal{I}^\vee \otimes p_X^*\omega_X[d] \right) \ast \mathcal{I} \xrightarrow{+1}
  \end{equation}
  with \cref{proposition:iso-1,proposition:iso-2}, we conclude that
  \begin{equation}
    \left( \mathcal{I}^\vee \otimes p_X^*\omega_X[d] \right) \ast \mathcal{I} \cong \mathcal{O}_{\Delta},
  \end{equation}
  and hence~$R\circ F \cong \identity_{\derived^\bounded(X)}$.
\end{proof}

Some remarks are in order.

\begin{remark}
  \label{remark:krug-ploog-sosna}
  This is not the first fully faithful functor from~$\derived^\bounded(X)$ to~$\derived^\bounded(X^{[2]})$. In \cite[theorem~4.1]{MR3811590} Krug--Ploog--Sosna constructed a semiorthogonal decomposition of the form
  \begin{equation}
    \derived^\bounded(\hilbtwo)
    =
    \left\langle
      \derived_{\mathbb{Z}/2\mathbb{Z}}^\bounded(X^2),
      \derived^\bounded(X),
      \ldots,
      \derived^\bounded(X)
    \right\rangle,
  \end{equation}
  where the first component is the equivariant derived category (or the derived category of the quotient stack~$[X^2/(\mathbb{Z}/2{Z})]$), and there are~$d-2$ copies of~$\derived^\bounded(X)$. For~$d=2$ this is an instance of the derived McKay correspondence of Bridgeland--King--Reid.

  Remark that this semiorthogonal decomposition is \emph{independent} of the geometry of~$X$, and in particular also works when~$\mathcal{O}_X$ is not an exceptional line bundle. But the fully faithful functors used in this decomposition are not related to the interpretation of~$\hilbtwo$ as a fine moduli space.

  For the remainder of this article we need the fully faithful functor to have a modular interpretation, as this is what allows us to compare the deformation theories of~$X$ and~$\hilbtwo$.
\end{remark}

\begin{remark}
  \label{remark:krug}
  In \cite[theorem~1.3]{1808.05931v1} it was shown that~$\Phi_{\mathcal{I}}$ (denoted in~$(-)^{[n]}$ in loc.~cit.) is always \emph{faithful}, independent of the properties of~$X$, when the number of points is subject to a numerical condition. This condition is satisfied for~$n=2$, so the existence of the left inverse in loc.~cit gives the claimed faithfulness.

  Notice that the faithfulness part in \cref{theorem:fully-faithful} follows in a straightforward way from the fullness of the Fourier--Mukai transform, as explained in the introduction of \cite{MR3054300}, and this is therefore the main new result in \cref{theorem:fully-faithful}.
\end{remark}

\begin{remark}
  In dimension~1,~$X$ is necessarily~$\mathbb{P}^1$, and the functor~$\Phi_{\mathcal{I}}$ is \emph{not} fully faithful. In this case~$\mathbb{P}^{1,[2]}\cong\Sym^2\mathbb{P}^1\cong\mathbb{P}^2$. In fact, there does not exist a fully faithful functor from~$\derived^\bounded(\mathbb{P}^1)$ to~$\derived^\bounded(\mathbb{P}^2)$, as there is no strong exceptional collection of length~2~in~$\derived^\bounded(\mathbb{P}^2)$ with~2\dash dimensional Hom-space. More precisely, the dimensions of the Hom-spaces between two exceptional objects on $\mathbb{P}^2$ are known, see \cite[proposition~8.1]{1805.00294v3} and \cite[example~3.2]{MR1230966}, and are given by the Markov numbers, i.e.~integers occurring in solutions of the Markov equation
  \begin{equation}
    x^2+y^2+z^2=xyz.
  \end{equation}
  As~2 is not a Markov number, it is indeed impossible to find a fully faithful functor from~$\derived^\bounded(\mathbb{P}^1)$ to~$\derived^\bounded(\mathbb{P}^2)$.
\end{remark}

\section{The relative local-to-global Ext spectral sequence}
\label{section:relative-ext}
In this section we prove \cref{theorem:B}. More precisely we prove the following statement.

\begin{theorem}
  \label{theorem:spectral-sequence}
  Let~$X$ be a smooth projective variety of dimension~$\geq 2$, such that~$\mathcal{O}_X$ is exceptional. Then there is a spectral sequence
  \begin{equation}
    \label{equation:spectral-sequence}
    \mathrm{E}_2^{i,j}=
    \begin{cases}
      k & \text{ if } i=j=0, \\
      0 & \text{ if } i \neq 0, j=0, \\
      \HH^i\left( \hilbtwo,q_*\left( \bigwedge^j\mathcal{N}_{Z/X\times \hilbtwo} \right) \right) & \text{ if } 1 \leq j \leq d-1, \\
      0 & \text{ if } j \geq d.
    \end{cases}
    \Rightarrow \HHHH^{i+j}(X)
  \end{equation}
  converging to the Hochschild cohomology of~$X$. This spectral sequences degenerates at the~$\mathrm{E}_2$-page and
  \begin{equation}
    \mathrm{E}_\infty^{i,j}=\mathrm{E}_2^{i,j}=
    \begin{cases}
      0 & \text{ if } j<0 \text{ or } j \geq d \\
      \HH^i(X,\bigwedge^j \tangent_X) & \text{ if } 1 \leq j \leq d-2 \\
      \text{an extension of } \HH^{i-1}(X,\omega_X^{-1}) \text{ by } \HH^i(X,\bigwedge^{d-1}\tangent_X) & \text{ if } j=d-1
    \end{cases}
  \end{equation}
\end{theorem}

To do so we will use the relative local-to-global Ext spectral sequence. The abutment of this spectral sequence is then closely related to the Hochschild cohomology of a variety as discussed in \cref{subsection:hochschild-cohomology}. In \cref{proposition:iso-tangent} we also give an interpretation to one of the sheaves appearing in the spectral sequence. This allows us to relate the deformation theories of~$X$ and~$\hilbtwo$ in \cref{subsection:dim-3-interpretation}.

We will briefly recall the notion of relative Ext in our setting. Let~$f\colon X\to Y$ be a morphism between smooth projective varieties, and~$\mathcal{E},\mathcal{F}$ are coherent sheaves on~$X\times Y$, which are flat over~$Y$. We will denote~$p_Y\colon X\times Y\to Y$ the second projection.
\begin{definition}
  The \emph{$j$th relative Ext} of~$\mathcal{E}$ and~$\mathcal{F}$ is the~$j$th derived functor of the composition~$p_{Y,*}\circ\sheafHom$, i.e.
  \begin{equation}
    \sheafExt_{p_Y}^j(\mathcal{E},\mathcal{F})\coloneqq\mathcal{H}^j\left( \RRR p_{Y,*}\RsheafHom(\mathcal{E},\mathcal{F}) \right)\cong\mathcal{H}^j\left( \RRR p_{Y,*}(\mathcal{E}^\vee\otimes^\LLL\mathcal{F}) \right),
  \end{equation}
  where~$(-)^\vee$ denotes the derived dual.
\end{definition}

We can use the Grothendieck spectral sequence to compute global Ext on~$X\times Y$, via the \emph{relative local-to-global Ext spectral sequence}
\begin{equation}
  \mathrm{E}_2^{i,j}=\HH^i(Y,\sheafExt_{p_Y}^j(\mathcal{E},\mathcal{F}))\Rightarrow\Ext_{X\times Y}^{i+j}(\mathcal{E},\mathcal{F}).
\end{equation}
This spectral sequence plays an important role in what follows.

\subsection{Relative Ext for the Hilbert square}
\label{subsection:relative-ext-hilb-square}
Just like before, unless specified otherwise, $X$ denotes a smooth projective variety of dimension~$d \geq 2$ with an exceptional structure sheaf. Recall from \eqref{equation:projection-diagram-notation} that~$p_{\hilbtwo}\colon X \times \hilbtwo \to \hilbtwo$ denotes the projection onto the second component. We will be interested in
\begin{equation}
  \sheafExt_{p_{\hilbtwo}}^j(\mathcal{I},\mathcal{I})\cong\mathcal{H}^j\left( \RRR p_{\hilbtwo,*}\left( \mathcal{I}^\vee\otimes^\LLL\mathcal{I} \right) \right).
\end{equation}
Then the relative local-to-global Ext spectral sequence for~$\mathcal{I} \in \coh X \times \hilbtwo$ takes the form
\begin{equation}
  \label{equation:l-t-g-ss}
  \mathrm{E}_2^{i,j}=\HH^i(\hilbtwo,\sheafExt^j_{p_{\hilbtwo}}(\mathcal{I},\mathcal{I})) \Rightarrow \Ext^{i+j}_{X \times \hilbtwo}(\mathcal{I},\mathcal{I}).
\end{equation}

To relate the Hochschild cohomology of~$X$ to the abutment of this spectral sequence we use the following standard lemma.

\begin{lemma}
  Let~$X_1,X_2,Y_1,Y_2$ be smooth projective varieties. If the functors $\Phi_{\mathcal{P}_i}\colon\derived^\bounded(X_i) \to \derived^\bounded(Y_i)$ are fully faithful~for~$i=1,2$, then the functor
  \begin{equation}
    \Phi_{\mathcal{P}_1 \boxtimes \mathcal{P}_2}\colon\derived^\bounded(X_1 \times X_2) \to \derived^\bounded(Y_1 \times Y_2)
  \end{equation}
  is also fully faithful.
\end{lemma}

\begin{proof}
  If~$\mathcal{P}_i^\RR$ denotes the kernel of the right adjoint to~$\Phi_{\mathcal{P}_i}$, then fully faithfulness of~$\Phi_{\mathcal{P}_i}$ is equivalent to the isomorphism
  \begin{equation}
    \Phi_{\mathcal{P}_i^\RR} \circ \Phi_{\mathcal{P}_i} \cong \Phi_{\mathcal{O}_{\Delta_i}}.
  \end{equation}
  One can check that the right adjoint to~$\Phi_{\mathcal{P}_1 \boxtimes \mathcal{P}_2}$ is given by $\Phi_{\mathcal{P}_1^\RR\boxtimes\mathcal{P}_2^\RR}$, and moreover
  \begin{equation}
    \begin{aligned}
      \Phi_{\mathcal{P}_1^\RR \boxtimes \mathcal{P}_2^\RR} \circ  \Phi_{\mathcal{P}_1 \boxtimes \mathcal{P}_2} &\cong  \Phi_{\mathcal{O}_{\Delta_1} \boxtimes \mathcal{O}_{\Delta_2}} \\
      &\cong \Phi_{\mathcal{O}_{\Delta_{1,2}}},
    \end{aligned}
  \end{equation}
  where $\Delta_{1,2}\colon X_1 \times X_2 \to (X_1 \times X_2) \times (X_1 \times X_2)$ is the diagonal morphism. Hence~$\Phi_{\mathcal{P}_1 \boxtimes \mathcal{P}_2}$ is also fully faithful.
\end{proof}

\begin{corollary}
  The functor
  \begin{equation}
    \Phi_{\mathcal{O}_{\Delta_X} \boxtimes \mathcal{I}}\colon\derived^\bounded(X \times X) \to \derived^\bounded(X \times \hilbtwo)
  \end{equation}
  is fully faithful.
\end{corollary}

\begin{proof}
  This is now immediate by the previous lemma and \cref{theorem:fully-faithful}.
\end{proof}

This functor sends~$\mathcal{O}_{\Delta_X}$ to~$\mathcal{I}$. Hence
\begin{equation}
  \HHHH^k(X)\cong\Ext_{X\times X}^k(\mathcal{O}_{\Delta_X},\mathcal{O}_{\Delta_X})\cong\Ext_{X\times\hilbtwo}^k(\mathcal{I},\mathcal{I})
\end{equation}
for all~$k\geq 0$, so the spectral sequence \eqref{equation:l-t-g-ss} abuts to the Hochschild cohomology of~$X$. This is the spectral sequence in \cref{theorem:spectral-sequence}, and we now proceed to analyse the terms on the second page more closely.

Recall from \eqref{equation:notation} that~$q\colon\Bl_\Delta(X \times X) \to \hilbtwo$ denotes the quotient~by $\mathbb{Z}/2\mathbb{Z}$. We can then prove the following theorem, which is part of the statement of \cref{theorem:spectral-sequence}.

\begin{theorem}
  \label{theorem:rel-ext}
  The relative Ext sheaves are given by
  \begin{equation}
    \label{equation:relative-ext}
    \mathcal{E}xt^j_{p_{\hilbtwo}}(\mathcal{I},\mathcal{I}) \cong
    \begin{cases}
      \mathcal{O}_{\hilbtwo} & \text{ if $j=0$}, \\
      q_*\left( \bigwedge^j\mathcal{N}_{Z/X \times \hilbtwo} \right) & \text{ if $j=1, \ldots, d-1$}, \\
      0 & \text{ if $j<0$ or $j \geq d$}
    \end{cases}
  \end{equation}
  where $\mathcal{N}_{Z/X \times \hilbtwo}$ denotes the normal bundle of the universal subscheme~$Z\cong \Bl_\Delta(X \times X)$ in $X \times \hilbtwo$.
\end{theorem}

\begin{proof}
  Taking the derived tensor product of the distinguished triangle associated to universal short exact sequence \eqref{equation:universal} with its dual distinguished triangle
  \begin{equation}
    (i_*\mathcal{O}_Z)^\vee\to \mathcal{O}_{X\times \hilbtwo}\to \mathcal{I}^\vee\xrightarrow{+1},
  \end{equation}
  we get the following diagram with all rows and columns distinguished triangles:
  \begin{equation}
    \label{equation:diag:TensorProduct}
    \begin{tikzcd}
      (i_*\mathcal{O}_Z)^\vee\otimes^\LLL\mathcal{I} \arrow[r] \arrow[d] & (i_*\mathcal{O}_Z)^\vee \arrow[r] \arrow[d] & (i_*\mathcal{O}_Z)^\vee\otimes^\LLL(i_*\mathcal{O}_Z) \arrow[r, "+1"] \arrow[d] & {} \\
      \mathcal{I} \arrow[r] \arrow[d] & \mathcal{O}_{X\times \hilbtwo} \arrow[r] \arrow[d] & i_*\mathcal{O}_Z \arrow[r, "+1"] \arrow[d] & {} \\
      \mathcal{I}^\vee\otimes^\LLL\mathcal{I} \arrow[r] \arrow[d, "+1"] & \mathcal{I}^\vee \arrow[r] \arrow[d, "+1"] & \mathcal{I}^\vee\otimes^\LLL(i_*\mathcal{O}_Z) \arrow[r, "+1"] \arrow[d, "+1"] & {} \\
      {} & {} & {}
    \end{tikzcd}.
  \end{equation}

  We can compute~$(i_*\mathcal{O}_Z)^\vee$ as in \eqref{equation:derived-dual-OZ}. Applying the derived pushforward functor $\RRR p_{\hilbtwo,*}$ to the diagram \eqref{equation:diag:TensorProduct}, together with the above computation of $(i_*\mathcal{O}_Z)^\vee$, we have
  \begin{equation}
    \label{equation:diag:PushTensorProduct}
    \begin{tikzcd}
      \RRR p_{\hilbtwo,*}\left((i_*\mathcal{O}_Z)^\vee\otimes^\LLL\mathcal{I} \right) \arrow[r] \arrow[d] & \RRR p_{\hilbtwo,*}\left( i_*\omega_i[-d] \right) \arrow[r] \arrow[d] & \RRR p_{\hilbtwo,*}\left( (i_*\mathcal{O}_Z)^\vee\otimes^\LLL i_*\mathcal{O}_Z \right) \arrow[r, "+1"] \arrow[d] & {} \\
      \RRR p_{\hilbtwo,*}\mathcal{I} \arrow[r] \arrow[d] & \RRR p_{\hilbtwo,*}\mathcal{O}_{X\times \hilbtwo} \arrow[r] \arrow[d] & \RRR p_{\hilbtwo,*}\left( i_*(\mathcal{O}_Z) \right) \arrow[d] \arrow[r, "+1"] & {} \\
      \RRR p_{\hilbtwo,*}\left(\mathcal{I}^\vee\otimes^\LLL\mathcal{I}\right) \arrow[r] \arrow[d, "+1"] & \RRR p_{\hilbtwo,*}(\mathcal{I}^\vee) \arrow[r] \arrow[d, "+1"] & \RRR p_{\hilbtwo,*}\left(\mathcal{I}^\vee\otimes^\LLL i_*\mathcal{O}_Z \right) \arrow[d, "+1"] \arrow[r, "+1"] & {} \\
      {} & {} & {}
    \end{tikzcd}
  \end{equation}

  In the following lemmas we will compute the square in the upper right more explicitly.

  \begin{lemma}
    \label{lemma:identification-morphism-1}
    We have isomorphisms
    \begin{equation}
      \RRR p_{\hilbtwo,*}\left( i_*\omega_i[-d] \right)\cong q_*\omega_i[-d],
    \end{equation}
    and
    \begin{equation}
      \RRR p_{\hilbtwo,*}\left( (i_*\mathcal{O}_Z)^\vee\otimes^\LLL i_*\mathcal{O}_Z \right)\cong q_*\left( \LLL i^*\circ i_*\mathcal{O}_Z\otimes \omega_i \right)[-d].
    \end{equation}
    Moreover, via these isomorphisms, the second morphism in the first horizontal triangle of \eqref{equation:diag:PushTensorProduct} is identified with the morphism induced by~$\mathcal{O}_Z=\LLL i^*\mathcal{O}_{X\times\hilbtwo}\to \LLL i^*\circ i_*(\mathcal{O}_Z)$, which comes from the restriction map~$\mathcal{O}_{X\times\hilbtwo}\to i_*\mathcal{O}_Z$.
  \end{lemma}

  \begin{proof}
    The first isomorphism is immediate from~$p_{\hilbtwo}\circ i=q$ in \eqref{equation:notation}.

    The second isomorphism follows from Grothendieck duality, the tensor-Hom adjunction and the commutative diagram \eqref{equation:notation}, as
    \begin{equation}
      \begin{aligned}
        \RRR p_{\hilbtwo,*}\left( (i_*\mathcal{O}_Z)^\vee\otimes^\LLL i_*\mathcal{O}_Z \right)
        &\cong \RRR p_{\hilbtwo,*}\RsheafHom\left(i_*\mathcal{O}_Z, i_*\mathcal{O}_Z\right) \\
        &\cong \RRR p_{\hilbtwo,*}\left( i_*\circ i^!\circ i_*(\mathcal{O}_Z) \right) \\
        &\cong q_*\left( \LLL i^*\circ i_*\mathcal{O}_Z\otimes \omega_i \right)[-d]
      \end{aligned}
    \end{equation}
  \end{proof}

  This lemma allows us to compute the cohomology of the object in the upper left corner of \eqref{equation:diag:PushTensorProduct}. Since~$i\colon Z \hookrightarrow X \times \hilbtwo$ is a closed immersion of smooth varieties, we know by \cite[proposition~11.8]{MR2244106} that~$\LLL i^*\circ i_*(\mathcal{O}_Z)$ has cohomology sheaves
  \begin{equation}
    \mathcal{H}^{-j}\left( \LLL i^*\circ i_*(\mathcal{O}_Z) \right)=\bigwedge^j \mathcal{N}_{Z/X\times\hilbtwo}^\vee,
  \end{equation}
  So by the identification of the morphism in \cref{lemma:identification-morphism-1} we can compute the cohomology sheaves of the object in the upper left corner as
  \begin{equation}
    \label{equation:OZI}
    \sheafExt^{j+1}_{p_{\hilbtwo}}\left(i_*\mathcal{O}_Z, \mathcal{I}\right)\cong
    \begin{cases}
      q_*\left( \bigwedge^{d-j}\mathcal{N}_{Z/X\times\hilbtwo} \right) & \text{ if $0\leq j\leq d-1$} \\
      0 & \text{ otherwise }
    \end{cases}
  \end{equation}
  because~$\omega_i\cong\det\mathcal{N}_{Z/X\times\hilbtwo}$, so~$\bigwedge^j\mathcal{N}_{Z/X\times\hilbtwo}^\vee\otimes \omega_i\cong \bigwedge^{d-j}\mathcal{N}_{Z/X\times\hilbtwo}$.

  Moreover we note that~$\sheafExt^{j+1}_{p_{\hilbtwo}}(i_*\mathcal{O}_Z, \mathcal{I})=0$ for~$j\geq d$, because the relative dimension of~$p_{\hilbtwo}$ is~$d$.

  The next lemma computes the second horizontal morphism in the upper right square of \eqref{equation:diag:PushTensorProduct} more explicitly.
  \begin{lemma}
    We have the isomorphisms
    \begin{equation}
      \RRR p_{\hilbtwo,*}\mathcal{O}_{X\times \hilbtwo}\cong \mathcal{O}_{\hilbtwo}.
    \end{equation}
    and
    \begin{equation}
      \RRR p_{\hilbtwo,*}\circ i_*(\mathcal{O}_Z)\cong q_*\mathcal{O}_Z\cong \mathcal{O}_{\hilbtwo}\oplus \mathcal{O}_{\hilbtwo}(-\delta),
    \end{equation}
    Moreover, via these isomorphisms, the second morphism in the second horizontal triangle of \eqref{equation:diag:PushTensorProduct} is the inclusion morphism
    \begin{equation}
      (\identity,0)\colon\mathcal{O}_{\hilbtwo}\hookrightarrow\mathcal{O}_{\hilbtwo}\oplus\mathcal{O}_{\hilbtwo}(-\delta).
    \end{equation}
  \end{lemma}

  \begin{proof}
    The first isomorphism follows from the isomorphism
    \begin{equation}
      \RRR p_{\hilbtwo,*}\mathcal{O}_{X\times \hilbtwo}\cong\HH^\bullet(X,\mathcal{O}_X) \otimes \mathcal{O}_{\hilbtwo}\cong \mathcal{O}_{\hilbtwo},
    \end{equation}
    whilst the second isomorphism follows from the commutativity of the diagram \eqref{equation:notation}.

    As for the last assertion, again by the commutativity of \eqref{equation:notation}, we see that the morphism~$\mathcal{O}_{\hilbtwo}\to q_*\mathcal{O}_Z$ coincides with the natural one induced by the morphism~$q$. Since the splitting~$q_*\mathcal{O}_Z\cong \mathcal{O}_{\hilbtwo}\oplus \mathcal{O}_{\hilbtwo}(-\delta)$ comes from the trace map~$q_*\mathcal{O}_Z\to \mathcal{O}_{\hilbtwo}$, after rescaling by~$\frac{1}{2}$, the structural morphism $\mathcal{O}_{\hilbtwo}\to q_*\mathcal{O}_Z$ is identified with the inclusion~$(\identity, 0)$.
  \end{proof}

  Therefore we can compute the first object in the triangle on the second row as
  \begin{equation}
    \label{equation:OI}
    \RRR p_{\hilbtwo,*}\mathcal{I}\cong \mathcal{O}_{\hilbtwo}(-\delta)[-1].
  \end{equation}

  Combining \eqref{equation:OZI} and \eqref{equation:OI} together with the distinguished triangle in the first column of \eqref{equation:diag:PushTensorProduct}, we get the desired formula \eqref{equation:relative-ext} except for the case~$j=0$. In that case, we have by the above computation that~$\RsheafHom_{p_{\hilbtwo}}(\mathcal{I}, \mathcal{I})$ is the kernel of the natural projection map
  \begin{equation}
    q_*\mathcal{O}_{\Bl_\Delta(X\times X)}\cong\mathcal{O}_{\hilbtwo}\oplus \mathcal{O}_{\hilbtwo}(-\delta)\xrightarrow{(0, \identity)} \mathcal{O}_{\hilbtwo}(-\delta),
  \end{equation}
  which is~$\mathcal{O}_{\hilbtwo}$.
\end{proof}

\paragraph{First relative Ext is tangent bundle}
The following result computes one of the sheaves in \eqref{equation:relative-ext}. Part of it is a generalisation of \cite[theorem~10.2.1]{MR2665168}, which treats the case~$d=2$. In this new generality we have not found a reference for it, so we prove it here.

\begin{proposition}
  \label{proposition:iso-tangent}
  There is an isomorphism of sheaves
  \begin{equation}
    \sheafExt^1_{p_{\hilbtwo}}(\mathcal{I},\mathcal{I})\cong q_*(\mathcal{N}_{Z/X\times\hilbtwo})\cong\tangent_{\hilbtwo}.
  \end{equation}
\end{proposition}

\begin{proof}
  The first isomorphism is proved in \cref{theorem:rel-ext}. Let us show that
  \begin{equation}
    \sheafExt^1_{p_{\hilbtwo}}(\mathcal{I},\mathcal{I})\cong\tangent_{\hilbtwo}.
  \end{equation}
  We claim that the Hilbert square~$X^{[2]}$ is isomorphic to~$\moduli(X, v)$, the (fine) moduli space of stable sheaves on~$X$ with Mukai vector
  \begin{equation}
    \label{eqn:MukaiVector}
    v\coloneqq\sqrt{\td(\tangent_X)}-(0,\ldots, 0, 2).
  \end{equation}
  Firstly, for any length-2 subscheme~$\xi$ of $X$, its ideal sheaf $\mathcal{I}_{\xi}$ is a rank~1 torsion-free (hence stable) sheaf with Mukai vector
  \begin{equation}
    v(\mathcal{I}_{\xi})=\left(1-\ch(\mathcal{O}_{\xi})\right)\cdot \sqrt{\td(\tangent_X)}=\sqrt{\td(\tangent_X)}-(0,\ldots, 0, 2)\cdot\sqrt{\td(\tangent_X)}=v.
  \end{equation}
  Conversely, for any stable sheaf~$E\in\moduli(X,v)$ with Mukai vector~$v$, its Chern character is~$\ch(E)=(1,0,\ldots, 0, -2)$. In particular, $E$ is of rank 1 and torsion-free. Moreover, its reflexive hull $E^{\vee\vee}$ is a line bundle which is of degree~0, as~$\mathrm{c}_1(E)=0$, hence trivial since by assumption~$\Pic^0(X)=0$. Therefore, $E$ is the ideal sheaf of a closed subscheme, denoted by~$\xi$. As~$\ch(\mathcal{O}_{\xi})=1-\ch(\mathcal{I}_{\xi})=(0,\ldots, 0, 2)$, the subscheme $\xi$ is 0-dimensional and of length~2. To summarize, we established an isomorphism from~$\hilbtwo$ to~$\moduli(X, v)$ by sending~$\xi$ to~$\mathcal{I}_{\xi}$.

  Now the desired isomorphism follows from the natural identification of~$\tangent_{\moduli(X,v)}$ and~$\sheafExt^1_{p_M}(\mathcal{I},\mathcal{I})$, where~$p_{\moduli(X,v})\colon\moduli(X,v)\times X\to \moduli(X,v)$ is the natural projection and~$\mathcal{I}$ is identified with the universal sheaf on~$\moduli(X,v)\times X$.
\end{proof}

\begin{remark}
  We do not have a modular interpretation for~$\sheafExt_{p_{\hilbtwo}}^j(\mathcal{I},\mathcal{I})$ with~$j\geq2$. When~$X$ is a surface we recall the description for~$j=2$ in \eqref{equation:huybrechts-lehn}.
\end{remark}

\subsection{Degeneration of the spectral sequence}
\label{subsection:degeneration}
In this section we finish the proof of \cref{theorem:spectral-sequence}. At this point we are left to show that the spectral sequence degenerates on the second page.

We start with a lemma from commutative algebra.
\begin{lemma}
  \label{lemma:Support}
  Let~$Y$ be a Noetherian integral scheme and~$i\colon D\hookrightarrow Y$ be the closed immersion of a hypersurface. Suppose there is a short exact sequence of coherent sheaves on $Y$
  \begin{equation}
    0\to\mathcal{F}_1\xrightarrow{f}\mathcal{F}_2\to i_*\mathcal{L}\to 0,
  \end{equation}
  where
  \begin{enumerate}
    \item $\mathcal{F}_1$ and~$\mathcal{F}_2$ are locally free sheaves on~$Y$, of rank~$r$;
    \item $\mathcal{L}$ is an invertible sheaf on~$D$.
  \end{enumerate}
  Then for every~$k\geq 0$ there exists a locally free sheaf~$\mathcal{E}$ on~$D$ of rank~${r}\choose{k-1}$, such that we have a short exact sequence
  \begin{equation}
    0\to\bigwedge^k\mathcal{F}_1\xrightarrow{\bigwedge^kf}\bigwedge^k\mathcal{F}_2\to i_*\mathcal{E}\to 0.
  \end{equation}
\end{lemma}

\begin{proof}
  The lemma being a local assertion, we can assume that~$Y=\Spec A$ for some local integral domain~$A$, and~$D=\Spec A/{(a)}$ for some~$a\neq 0$, and~$i$ is induced by the projection~$A\twoheadrightarrow A/{(a)}$. We identify~$\mathcal{F}_1$ and~$\mathcal{F}_{2}$ (resp.~$\mathcal{L}$) with the corresponding free~$A$\dash modules~$F_1$ and~$F_2$ (resp.~$A/{(a)}$ itself). Then we have a short exact sequence of~$A$\dash modules
  \begin{equation}
    0\to F_1\xrightarrow{f}F_2\to A/{(a)}\to 0,
  \end{equation}
  which is a free resolution of~$A/{(a)}$.

  Using~\cite[theorem 20.2]{MR1322960} we have that~$F_1\cong F_1'\oplus A$ and~$F_2\cong F_2'\oplus A$, such that~$f=(f', \cdot a)$ for some free~$A$\dash modules~$F_1'$, $F_2'$ and an isomorphism~$f'$ between them. Therefore the morphism
  \begin{equation}
    \bigwedge^kF_1\xrightarrow{\bigwedge^kf}\bigwedge^kF_{2}
  \end{equation}
  can be written under the natural identifications as
  \begin{equation}
    \bigwedge^kF_1'\oplus\bigwedge^{k-1}F_1'\otimes_A A\xrightarrow{(\bigwedge^kf', \bigwedge^{k-1}f'\otimes a)}\bigwedge^kF_2'\oplus\bigwedge^{k-1}F_2'\otimes_A A.
  \end{equation}
  As~$\bigwedge^kf'$ and~$\bigwedge^{k-1}f'$ are isomorphisms, we see that the cokernel is a free~$A/{(a)}$-module of rank~$r \choose{k-1}$.
\end{proof}

The following lemma is crucial to the proof. In \cref{lemma:Z-to-X,lemma:wedge-vanishing} we will give an interpretation of the sheaf cohomology (on~$Z$) of the outer terms in the sequence \eqref{equation:wedge-seq} using sheaf cohomology of polyvector fields on~$X$.

\begin{lemma}
  For any~$k \geq 1$, there is a short exact sequence
  \begin{equation}
    \label{equation:wedge-seq}
    0 \to \pi^*\circ p_1^*\left(\bigwedge^k\tangent_X\right) \to \bigwedge^k\mathcal{N}_{Z/X\times\hilbtwo} \to j_*\left(\bigwedge^{k-1}\tangent_p(-k-1)\right) \to 0.
  \end{equation}
\end{lemma}

\begin{proof}
  Since~$\mathcal{N}_{X/X \times X}\cong \tangent_X$ and~$p\colon E\cong\mathbf{P}(\mathcal{N}_{X/X \times X}) \to X$ is a projective bundle, the corresponding relative Euler sequence takes the form
  \begin{equation}
    \label{equation:rel-euler}
    0 \to \mathcal{O}_p(-1) \to p^*\tangent_X \to \tangent_p(-1) \to 0.
  \end{equation}
  Now for~$k \geq 1$, taking the~$k$th exterior power of \eqref{equation:rel-euler} one obtains sequences
  \begin{equation}
    \label{equation:wedge-euler}
    0 \to\left( \bigwedge^{k-1}\tangent_p \right)(-k) \to p^*\left( \bigwedge^k\tangent_X \right)\to\left( \bigwedge^k \tangent_p \right)(-k) \to 0.
  \end{equation}

  Observe that the restriction of~$\tangent_{X\times \hilbtwo}$ to~$Z$ is isomorphic to~$\pi^*\circ p_1^*(\tangent_X)\oplus q^*\tangent_{\hilbtwo}$, hence we have the following commutative diagram with rows being short exact sequences:
  \begin{equation}
    \begin{tikzcd}
      0 \arrow[r] & \tangent_Z \arrow[d, equal] \arrow[r] & \pi^*\circ p_1^*(\tangent_X)\oplus q^*\tangent_{\hilbtwo} \arrow[d, two heads, "{(0,\identity)}"] \arrow[r] & \mathcal{N}_{Z/X\times\hilbtwo} \arrow[r] \arrow[d] & 0 \\
      0 \ar{r} & \tangent_Z \ar{r}& q^*\tangent_{\hilbtwo}\ar{r} & j_*\mathcal{O}_{E}(2E) \ar{r}& 0
    \end{tikzcd}
  \end{equation}
  By the snake lemma we get a short exact sequence
  \begin{equation}
    \label{equation:N}
    0\to \pi^*\circ p_1^*(\tangent_X)\to\mathcal{N}_{Z/X\times\hilbtwo} \to j_*\left( \mathcal{O}_p(-2) \right)\to 0
  \end{equation}
  which is the~$k=1$ case of \eqref{equation:wedge-seq}.

  Applying~$\LLL j^*\circ q_*$ to \eqref{equation:N} and taking cohomology we have the exact sequence
  \begin{equation}
    \label{equation:some-seq}
    0\to \mathcal{O}_p(-1)\to p^*\tangent_X\to j^*\mathcal{N}_{Z/X\times\hilbtwo}\to \mathcal{O}_p(-2)\to 0,
  \end{equation}
  where we used \cite[corollary~11.2]{MR2244106} and the fact that~$p_1\circ \pi\circ j=p_1\circ\Delta\circ p=p$ from \eqref{equation:notation}. One checks that the first morphism in \eqref{equation:some-seq} is the one in \eqref{equation:rel-euler}, and hence we have a short exact sequence:
  \begin{equation}
    \label{equation:jN}
    0\to \tangent_p(-1)\to j^*\mathcal{N}_{Z/X\times\hilbtwo}\to \mathcal{O}_p(-2)\to 0.
  \end{equation}
  Take the~$k$th exterior power of \eqref{equation:jN} to obtain:
  \begin{equation}
    \label{equation:jwedgeN}
    0\to \bigwedge^k\tangent_p(-k)\to j^*\left(\bigwedge^k\mathcal{N}\right)\to \bigwedge^{k-1}\tangent_p(-k-1)\to 0.
  \end{equation}
  Putting \eqref{equation:jwedgeN} and \eqref{equation:wedge-euler} together, we obtain the exact sequence
  \begin{equation}\label{equation:DerivedRestriction}
    0\to \bigwedge^{k-1}\tangent_p(-k)\to p^*\left(\bigwedge^k\tangent_X\right)\to j^*\left(\bigwedge^k\mathcal{N}_{Z/X\times\hilbtwo}\right)\to \bigwedge^{k-1}\tangent_p(-k-1)\to 0,
  \end{equation}
  which says precisely that if we apply~$\LLL j^*$ to \eqref{equation:wedge-seq}, we have a distinguished triangle.

  On the other hand, by \cref{lemma:Support} applied to \eqref{equation:N}, there exists a locally free sheaf~$\mathcal{F}$ on~$E$ fitting into the following short exact sequence:
  \begin{equation}
    \label{equation:desired}
    0 \to \pi^*\circ p_1^*\left( \bigwedge^k\tangent_X \right) \to \bigwedge^k\mathcal{N}_{Z/X\times\hilbtwo} \to j_*\mathcal{F} \to 0.
  \end{equation}
  Applying~$\LLL j^*$ to it and comparing to \eqref{equation:DerivedRestriction}, we see that~$\mathcal{F}\cong \bigwedge^{k-1}\tangent_p(-k-1)$. Hence \eqref{equation:desired} is the desired exact sequence.
\end{proof}

We now interpret the sheaf cohomology of the outer terms of \eqref{equation:wedge-seq}. The first term is dealt with by the following lemma.
\begin{lemma}
  \label{lemma:Z-to-X}
  There are isomorphisms
  \begin{equation}
    \HH^i\left( Z,\pi^*\circ p_1^*\left( \bigwedge^j\tangent_X \right) \right)\cong\HH^i\left( X,\bigwedge^j\tangent_X \right).
  \end{equation}
  for all~$i\geq 0$ and~$j\geq 0$.
\end{lemma}

\begin{proof}
  By the projection formula, the isomorphism~$\RRR\pi_*\mathcal{O}_Z\cong\mathcal{O}_X$ and the vanishing~$\HH^{\geq 1}(X,\mathcal{O}_X)=0$ we have that
  \begin{equation}
    \begin{aligned}
      \HH^i\left( Z,\pi^*\circ p_1^*\left( \bigwedge^j\tangent_X \right) \right)
      &\cong\HH^i\left( X\times X,p_1^*\left( \bigwedge^j\tangent_X \right) \right) \\
      &\cong\HH^i\left( X,\bigwedge^j\tangent_X \right).
    \end{aligned}
  \end{equation}
\end{proof}

For the third term we can prove the following, where we have already applied the adjunction~$\LLL j^*\dashv j_*$ in computing the sheaf cohomology.
\begin{lemma}
  \label{lemma:wedge-vanishing}
  There are isomorphisms
  \begin{equation}
    \HH^i\left( E,\bigwedge^{j-1}\tangent_p(-j-1) \right)\cong
    \begin{cases}
      0 & j \leq d-2 \\
      \HH^i(X,\omega_X^{-1}) & j=d-1
    \end{cases}
  \end{equation}
  for all~$i\geq 0$.
\end{lemma}

\begin{proof}
  For any $l \in \mathbb{Z}$, twisting \eqref{equation:wedge-euler} by $\mathcal{O}_p(-l)$ one obtains sequences
  \begin{equation}
    \label{equation:wedge-euler-l}
    0 \to \bigwedge^{k-1}\tangent_p(-k-l) \to p^*\bigwedge^k\tangent_X \otimes \mathcal{O}_p(-l) \to \bigwedge^k \tangent_p(-k-l) \to 0.
  \end{equation}

  Let us consider~$l=1,\ldots,d-1$. In this case we have~$\RRR p_*\mathcal{O}_p(-l)=0$, and hence
  \begin{equation}
    \HH^\bullet\left( E,p^*\bigwedge^k\tangent_X \otimes \mathcal{O}_p(-l) \right)=0.
  \end{equation}
  So from the long exact sequence for sheaf cohomology obtained from \eqref{equation:wedge-euler} we find isomorphisms
  \begin{equation}
    \HH^i\left( E,\bigwedge^k\tangent_p(-k-l) \right)\cong\HH^{i+1}\left( E,\bigwedge^{k-1}\tangent_p(-k-l) \right),
  \end{equation}
  for all~$i \in \mathbb{Z}$. Repeated use of this for different values of~$k$ gives
  \begin{equation}
    \HH^i\left( E,\bigwedge^{j-1}\tangent_p(-j-1) \right)
    \cong\HH^{i+j+1}(E,\mathcal{O}_p(-j-1))
    \cong
    \begin{cases}
      0 & j \neq d-1 \\
      \HH^{i+d-2}(E,\mathcal{O}_p(-d)) & j=d-1
    \end{cases}
  \end{equation}
  for all~$i \in \mathbb{Z}$.

  Finally, by relative Serre duality~\cite{MR0578050}
  \begin{equation}
    \RRR p_*\left(\mathcal{O}_p(-d)\right)\cong\RR^{d-1}p_*\left( \mathcal{O}_p(-d) \right)[1-d]\cong\omega_X^{-1}[1-d],
  \end{equation}
  and so
  \begin{equation}
    \HH^{i+d-2}(E,\mathcal{O}_p(-d)) \cong \HH^{i-1}(X,\omega_X^{-1}),
  \end{equation}
  which proves the lemma.
\end{proof}

We can now show the degeneration of the spectral sequence, which is second part of \cref{theorem:spectral-sequence}.
\begin{theorem}
  \label{theorem:degeneration-again}
  The spectral sequence \eqref{equation:spectral-sequence} degenerates at~$\mathrm{E}_2$.
\end{theorem}

\begin{proof}
  We know that
  \begin{equation}
    \label{equation:degeneration}
    \dim_k \HHHH^n(X)=\sum_{i+j=n}\dim_k \mathrm{E}_{\infty}^{i,j} \leq \sum_{i+j=n} \dim_k \mathrm{E}_2^{i,j},
  \end{equation}
  and to establish degeneration it suffices to prove that the last inequality is in fact an equality. By the Hochschild--Kostant--Rosenberg theorem we know that
  \begin{equation}
    \dim_k \HHHH^n(X)=\sum_{i+j=n}\dim_k \HH^i\left( X,\bigwedge^j\tangent_X \right),
  \end{equation}
  so degeneration would follow by establishing the inequalities
  \begin{equation}
    \label{equation:inequalities}
    \sum_{i+j=n} \dim_k \mathrm{E}_2^{i,j} \leq \sum_{i+j=n} \dim_k \HH^i(X,\bigwedge^j \tangent_X),
  \end{equation}
  for all~$n \geq 0$. This is what we will do now.

  %In the sequel, we denote $\mathcal{N}:=\mathcal{N}_{Z/X\times\hilbtwo}$.

  Since~$q_*$ is an exact functor, we have for every~$j=1,\ldots,d-1$
  \begin{equation}
    \mathrm{E}_2^{i,j}=\HH^i\left(\hilbtwo, q_*\left( \bigwedge^j\mathcal{N}_{Z/X\times\hilbtwo} \right) \right)\cong \HH^i\left(Z, \bigwedge^j\mathcal{N}_{Z/X\times\hilbtwo}\right).
  \end{equation}
  Now using the long exact sequence associated to \eqref{equation:wedge-seq} for $k=j$ we get
  \begin{equation}
    \label{equation:long-exact}
    \cdots \to \HH^i\left(Z,\pi^*p_1^*\bigwedge^j\tangent_X\right) \to \HH^i\left(Z,\bigwedge^j\mathcal{N}_{Z/X\times\hilbtwo}\right) \to \HH^i\left(Z,j_*\bigwedge^{j-1}\tangent_p(-j-1)\right) \to \cdots
  \end{equation}

  Combining \eqref{equation:long-exact} with \cref{lemma:Z-to-X,lemma:wedge-vanishing}, we have for~$j=1,\ldots,d-2$ that
  \begin{equation}
    \label{equation:LowerLinesInSS}
    \mathrm{E}_2^{i,j}=\HH^i\left(Z, \bigwedge^j\mathcal{N}_{Z/X\times\hilbtwo}\right)\cong \HH^i\left(X, \bigwedge^j\tangent_X\right).
  \end{equation}

  On the other hand, for~$j=d-1$ we have that~$\mathrm{E}_2^{i,d-1}=\HH^i\left(Z, \bigwedge^{d-1}\mathcal{N}_{Z/X\times\hilbtwo}\right)$ fits into a long exact sequence with parts
  \begin{equation}
    \label{equation:TopLineInSS}
    \cdots \HH^{i-2}\left(X, \omega_X^{-1}\right)\to\HH^i\left(X, \bigwedge^{d-1}\tangent_X\right)\to \mathrm{E}_2^{i,d-1}\to\HH^{i-1}\left(X, \omega_X^{-1}\right)\to\HH^{i+1}\left(X, \bigwedge^{d-1}\tangent_X\right)\to\cdots
  \end{equation}
  hence
  \begin{equation}
    \dim_k\mathrm{E}_2^{i,d-1}\leq\dim_k\HH^i\left( X,\bigwedge^{d-1}\tangent_X \right)+\dim_k\HH^{i-1}\left( X,\omega_X^{-1} \right)
  \end{equation}
  which suffices to obtain \eqref{equation:inequalities}, and hence degeneration.
\end{proof}

\begin{corollary}
  \label{corollary:short-exact}
  For every~$i \in \mathbb{Z}$ there are short exact sequences
  \begin{equation}
    0\to \HH^i\left(X, \bigwedge^{d-1}\tangent_X\right)\to \mathrm{E}_2^{i,d-1}\to\HH^{i-1}\left(X, \omega_X^{-1}\right)\to 0.
  \end{equation}
\end{corollary}

\begin{proof}
  This follows immediately from the inequality \eqref{equation:inequalities}, which implies that the connecting morphisms
  \begin{equation}
    \HH^{i-2}(X, \omega_X^{-1})\to\HH^i\left(X, \bigwedge^{d-1}\tangent_X\right)
  \end{equation}
  in \eqref{equation:TopLineInSS} are zero.
\end{proof}

\subsection{Degeneration of the spectral sequence for surfaces}
\label{subsection:degeneration-surfaces}
In dimension~2 the degeneration can be understood more directly as follows, where we will use~$S$ instead of~$X$ to emphasise that we are working in dimension~2. Note that we do not need to restrict ourselves to~2~points, as~$\hilbn{n}{S}$ is smooth for every~$n\geq 2$, and we can use the fully faithfulness result by Krug--Sosna, as long as~$\mathcal{O}_S$ is exceptional.

We want to describe the relative local-to-global Ext spectral sequence
\begin{equation}
  \mathrm{E}_2^{i,j}=\HH^i(\hilbn{n}{S},\sheafExt_{p_{\hilbn{n}{S}}}^j(\mathcal{I},\mathcal{I}))\Rightarrow\Ext_{S\times\hilbn{n}{S}}^{i+j}(\mathcal{I},\mathcal{I}),
\end{equation}
where $p_{\hilbn{n}{S}}: S\times \hilbn{n}{S}\to \hilbn{n}{S}$ is the natural projection.
Notice that in this case~$\sheafExt_{p_{\hilbn{n}{S}}}^j(\mathcal{I},\mathcal{I})=0$
if~$j\geq 3$, as the relative dimension of~$p_{\hilbn{n}{S}}$ is~2.

In \cref{proposition:iso-tangent} we have described the relative Ext when~$j=1$. The only mising case is~$j=2$, and as~$\hilbn{n}{S}$ is the moduli space of stable sheaves with Mukai vector~$\sqrt{\td_S}\cdot(1,0,-n)=\sqrt{\td_{S}}-(0,0,n)$, we can apply \cite[theorem~10.2.1]{MR2665168} to get
\begin{equation}
  \label{equation:huybrechts-lehn}
  \sheafExt_{p_{\hilbn{n}{S}}}^j(\mathcal{I},\mathcal{I})\cong
  \begin{cases}
    \mathcal{O}_{\hilbn{n}{S}} & j=0 \\
    \tangent_{\hilbn{n}{S}} & j=1 \\
    \HH^2(S,\mathcal{O}_S)\otimes_k\mathcal{O}_{\hilbn{n}{S}} & j=2
  \end{cases}
\end{equation}
In particular we get that in the situation we are mostly interested in, i.e.~where~$\mathcal{O}_S$ is exceptional, the interesting part of the~$\mathrm{E}_2$\dash page therefore looks like
\begin{equation}
  \label{equation:E2-surface}
  \begin{tikzcd}[column sep=tiny]
    0 \arrow[rrd] & 0 \arrow[rrd] & 0 & 0 & \ldots \\
    \HH^0(\hilbn{n}{S},\tangent_{\hilbn{n}{S}}) \arrow[rrd] & \HH^1(\hilbn{n}{S},\tangent_{\hilbn{n}{S}}) \arrow[rrd] & \HH^2(\hilbn{n}{S},\tangent_{\hilbn{n}{S}}) & \HH^3(\hilbn{n}{S},\tangent_{\hilbn{n}{S}}) & \ldots \\
    \HH^0(\hilbn{n}{S},\mathcal{O}_{\hilbn{n}{S}}) & 0 & 0 & 0 & \ldots
  \end{tikzcd}
\end{equation}
using \eqref{equation:huybrechts-lehn} and \cref{lemma:HodgeNumberHilb}. All the morphisms are necessarily~0. By the following proposition we know that the second row is moreover~0 to the right of what is pictured.
\begin{proposition}
  \label{proposition:h-hilb-hh-surface}
  Let~$S$ be a smooth projective surface with exceptional structure sheaf. Then
  \begin{equation}
    \HH^i(\hilbn{n}{S},\tangent_{\hilbn{n}{S}})\cong
    \begin{cases}
      \HHHH^{i+1}(S) & i\leq 3 \\
      0 & i\geq 4.
    \end{cases}
  \end{equation}
\end{proposition}

\begin{proof}
  As the spectral sequence degenerates, we have that
  \begin{equation}
    \mathrm{E}_2^{i,1}\cong\Ext_{S\times\hilbn{n}{S}}^{i+1}(\mathcal{I},\mathcal{I})\cong\HHHH^{i+1}(S),
  \end{equation}
  where the second isomorphism uses the fully faithfulness (\cref{theorem:fully-faithful}). By the Hochschild--Kostant--Rosenberg decomposition we know that~$\HHHH^{i+1}(S)=0$ for~$i\geq 4$.
\end{proof}

\begin{remark}
  Nowhere have we used that the moduli space is the Hilbert scheme of points. The identification \eqref{equation:huybrechts-lehn} and the relative local-to-global Ext spectral sequence can be used verbatim for other moduli spaces of sheaves on surfaces with exceptional structure sheaf. The only required ingredient is the fully faithful functor
  \begin{equation}
    \Phi_{\mathcal{E}}\colon\derived^\bounded(S)\to\derived^\bounded(\moduli_S(r,c_1,c_2))
  \end{equation}
  where~$(r,c_1,c_2)$ is such that~$\moduli_S(r,c_1,c_2)$ is smooth and projective, and~$\mathcal{E}$ denotes the universal sheaf on~$S\times\moduli_S(r,c_1,c_2)$. Unfortunately we have no other examples of this phenomenon.
\end{remark}

Let us give an example of a surface such that~$\HH^3(\hilbn{n}{S},\tangent_{\hilbn{n}{S}})$ is nonzero.
\begin{example}
  \label{example:fake-projective-plane}
  Let~$S$ be a fake projective plane. As the Hodge numbers of~$S$ are the same as those of~$\mathbb{P}^2$ we know that the structure sheaf is an exceptional object. For the fake projective planes in \cite{MR3341791} we have by remark~3.2 of~op.cit.~that
  \begin{equation}
    \dim_k\HH^2(S,\omega_S^\vee)=9.
  \end{equation}
  Hence~$\HH^3(\hilbn{n}{S},\tangent_{\hilbn{n}{S}})\cong \HHHH^{4}(S)\cong\HH^2(S, \omega_{S}^{\vee})\neq 0$ in this case.
\end{example}

On the other hand, the vanishing for~$i\geq 4$ does not necessarily hold when the structure sheaf is not exceptional.
\begin{example}
  \label{example:K3-1}
  Let~$S$ be a K3~surface. As~$\hilbn{n}{S}$ is hyperk\"ahler we can see using \cite{MR1219901} that
  \begin{equation}
    \dim_k\HH^{2n-1}(\hilbn{n}{S},\tangent_{\hilbn{n}{S}})=21.
  \end{equation}
  A more careful analysis shows that it is nonzero in all odd degrees up to~$2n-1$.
\end{example}

\section{Deformations of Hilbert squares}
\label{section:deformation-hilbert-squares}
The deformation theory of Hilbert schemes of points on curves and surfaces is well-studied. The case of curves, to which our results do not apply, can be found in \cite{MR1272697}. For surfaces the main results can be found in \cite{MR1354269}, and we will briefly summarise the situation before explaining how our results give new proofs for some known results.

In higher dimensions the description of the deformation theory of the Hilbert square is new, as far as we know. It is also somewhat surprising, as it gives a strong no-go result for noncommutative algebraic geometry when compared to the case of noncommutative surfaces, as we will explain.

\subsection{Dimension 2}
\label{subsection:dimension-2}
To emphasise that we are in the~2\dash dimensional setting here, we will write~$S$ instead of~$X$ to denote a smooth projective surface, as we did in \cref{subsection:degeneration-surfaces}. The results of \cite{MR1354269} can be summarised as follows.
\begin{theorem}[Fantechi]
  \label{theorem:fantechi}
  Let~$S$ be a smooth projective surface such that
  \begin{enumerate}
    \item $\HH^0(S,\tangent_S)=0$ or~$\HH^1(S,\mathcal{O}_S)=0$;
    \item and~$\HH^0(S,\omega_S^\vee)=0$
  \end{enumerate}
  then the natural transformation
  %\footnote{Obtained by composing the natural maps from deformations of~$S$ to locally trivial deformations of the singular variety~$\Sym^rS$ to deformations of its desingularisation~$\hilbn{n}{S}$.}
  \begin{equation}
    \label{eq:nat-eq}
    \Def_S\to\Def_{\hilbn{n}{S}}
  \end{equation}
  is a natural equivalence of deformation functors.
\end{theorem}

In particular, when~$S$ is a surface of general type, or an Enriques surface, we have that~$\Def_S\cong\Def_{\hilbn{n}{S}}$. On the other hand, for K3~surfaces~$\dim_k\HH^1(\hilbn{n}{S},\tangent_{\hilbn{n}{S}})=21$, whereas~$\dim_k\HH^{1}(S,\tangent_S)=20$, so \eqref{eq:nat-eq} cannot be an equivalence (see also \cref{example:K3-1}).

For del Pezzo surfaces the deformation theory also involves \emph{noncommutative deformations}, since the~$\HH^0(S,\omega_S^\vee)$ term appearing in the Hochschild--Kostant--Rosenberg decomposition for~$\HHHH^2(S)$ is non-zero. This leads to a very rich theory: for~$\mathbb{P}^2$ and~$\mathbb{P}^1\times\mathbb{P}^1$ consult \cite{MR1230966,MR2836401}, and for their blowups see \cite{MR1846352}. As will be explained below this is reflected in the non-rigidity of the corresponding Hilbert schemes of points \cite{MR2303228,MR2102090}.

From now on we will assume that~$\mathcal{O}_S$ is an exceptional line bundle, so that we can apply the fully faithfulness result of Krug--Sosna for~$n\geq 2$ (resp.~ours for~$n=2$). We want to keep~2~types of examples in mind:
\begin{enumerate}
  \item surfaces of general type with exceptional structure sheaf, such as fake projective planes (see also \cref{example:fake-projective-plane}), Beauville surfaces (or fake projective quadrics), or Burniat surfaces;
  \item del Pezzo surfaces.
\end{enumerate}
Remark that we can only apply Fantechi's result to the first type of surfaces, as the second condition in theorem \ref{theorem:fantechi} rules out Poisson surfaces.

We have that \cref{corollary:short-exact} specialises to give the short exact sequences
\begin{equation}
  \label{equation:hitchin-short}
  0\to \HH^i(S,\tangent_S)\to \HH^i(\hilbn{n}{S},\tangent_{\hilbn{n}{S}})\to\HH^{i-1}(S, \omega_S^{-1})\to 0.
\end{equation}
For~$i=1$ this short exact sequence can be compared to a short exact sequence \eqref{equation:hitchin-short-exact} obtained by Hitchin using purely geometric methods, not involving any derived categories \cite[theorem~9]{MR3024823}. For completeness' sake we give the full statement, including the (weaker than ours) condition on~$S$.
\begin{theorem}[Hitchin]
  \label{theorem:hitchin}
  Let~$S$ be a smooth projective surface such that~$\HH^1(S,\mathcal{O}_S)=0$. Then there exists a short exact sequence
  \begin{equation}
    0\to\HH^1(S,\tangent_S)\to\HH^1(\hilbn{n}{S},\tangent_{\hilbn{n}{S}})\to\HH^0(S,\omega_S^\vee)\to 0.
  \end{equation}
\end{theorem}
This sequence expresses that the deformations of~$\hilbn{n}{S}$ are described as a combination of geometric deformations and noncommutative deformations of~$S$. Coming back to the~2~types of surfaces we are interested in, we have that
\begin{enumerate}
  \item deformations of~$\hilbn{n}{S}$ when~$S$ is a surface of general type with exceptional structure sheaf are all induced by geometric deformations of~$S$, as explained by Fantechi's result in \cref{theorem:fantechi}, because there are no Poisson structures on such surfaces;
  \item deformations of~$\hilbn{n}{S}$ when~$S$ is a del Pezzo surface are indeed described using (possibly) noncommutative deformations of~$S$, as explained in \cite{MR2303228,MR2102090,MR3669875}.
\end{enumerate}

\begin{remark}
  \label{remark:noncommutative-deformations}
  In \cite{MR3488782,belmans-raedschelders} the fully faithful functor~$\derived^\bounded(S)\to\derived^\bounded(\hilbn{2}{S})$ is generalised to \emph{noncommutative deformations} of~$S=\mathbb{P}^2,\mathbb{P}^1\times\mathbb{P}^1$. The Hilbert scheme of~$2$~points has commutative deformations, as one can see from \cref{corollary:short-exact} for~$i=2$ or \cref{theorem:hitchin}. For a noncommutative plane or quadric, denoted~$S_{\mathrm{nc}}$, the associated commutative deformation of~$\hilbn{2}{S}$ will be denoted~$\hilbn{2}{S_{\mathrm{nc}}}$. In op.~cit.~the fully faithful functor~$\Phi_{\mathcal{I}}$ is constructed in families of noncommutative projective planes or quadrics.

  The isomorphism
  \begin{equation}
    \HHHH^i(S)\cong\HH^{i-1}(\hilbn{2}{S},\tangent_{\hilbn{2}{S}})
  \end{equation}
  should become an isomorphism
  \begin{equation}
    \HHHH^i(S_{\mathrm{nc}})\cong\HH^{i-1}(\hilbn{2}{S_{\mathrm{nc}}},\tangent_{\hilbn{2}{S_{\mathrm{nc}}}}).
  \end{equation}
  Evidence for this is the dimension drop in the Sklyanin case when~$S=\mathbb{P}^2$, using \cite[proposition~12]{MR3024823} and the description of the Hochschild cohomology of noncommutative planes from \cite{1705.06098v1}.
\end{remark}

\subsection{Dimension \texorpdfstring{$\geq 3$}{at least 3}}
\label{subsection:dim-3-interpretation}
The following theorem can be deduced from the previous section, but we give a self-contained proof. We keep the notation from \eqref{equation:notation}.

\begin{theorem}
  \label{theorem:main-comparison}
  Let~$X$ be a smooth projective variety of dimension~$d$ at least~3, such that~$\mathcal{O}_X$ is exceptional. Then there are isomorphisms
  \begin{equation}
    \HH^i(X,\tangent_X)\overset{\cong}{\to}\HH^i(\hilbtwo,\tangent_{\hilbtwo})
  \end{equation}
  for all~$i\geq 0$.
\end{theorem}

\begin{proof}
  We first claim that
  \begin{equation}
    \label{equation:degeneration-isomorphism}
    \RRR\pi_*\mathcal{N}_{Z/X\times\hilbtwo}\cong p_1^*\tangent_X.
  \end{equation}
  To see this, recall that there is an exact sequence
  \begin{equation}
    0 \to \pi^*\circ p_1^*(\tangent_X) \to \mathcal{N}_{Z/X\times\hilbtwo} \to j_*\mathcal{O}_E(2E) \to 0,
  \end{equation}
  which was constructed in \eqref{equation:N}. Applying~$\RRR\pi_*$ to this sequence, using that~$\RRR\pi_*\mathcal{O}_Z \cong \mathcal{O}_{X \times X}$ and the projection formula, we obtain a short exact sequence
  \begin{equation}
    \label{equation:ses}
    0 \to p_1^*\tangent_X \to \pi_*\mathcal{N}_{Z/X\times\hilbtwo} \to \pi_*\circ j_*(\mathcal{O}_E(2E)) \to 0
  \end{equation}
  and isomorphisms
  \begin{equation}
    \RR^i\pi_*\mathcal{N}_{Z/X\times\hilbtwo} \cong \RR^i\pi_*\circ j_*(\mathcal{O}_E(2E)),
  \end{equation}
  for~$i\geq 1$. It now suffices to compute that in the fibers
  \begin{equation}
    (\RR^i\pi_*\circ j_*(\mathcal{O}_E(2E)))_x \cong \HH^i(\mathbb{P}^{d-1},\mathcal{O}_{\mathbb{P}^{d-1}}(-2))=0,
  \end{equation}
  for~any $x\in X$ and~$i \geq 0$. Note that the last equality uses the assumption that~$d$ is at least~3.

  The Leray spectral sequence for the blowup morphism~$\pi$ reads
  \begin{equation}
    \label{equation:spec-leray}
    \mathrm{E}_2^{i,j}=\HH^i(X \times X,\RR^j\pi_*\mathcal{N}_{Z/X\times\hilbtwo}) \Rightarrow \HH^{i+j}(Z,\mathcal{N}_{Z/X\times\hilbtwo})
  \end{equation}
  and hence degenerates at the second page by \eqref{equation:degeneration-isomorphism}. We obtain isomorphisms
  \begin{equation}
    \HH^i(X,\tangent_X)\cong \HH^i(X \times X,p_1^*\tangent_X) \xrightarrow{\cong} \HH^i(Z,\mathcal{N}_{Z/X\times\hilbtwo}),
  \end{equation}
  for~$i \geq 0$, using the K\"unneth formula for the first isomorphism. Since~$q_{*}$ is exact, and using \cref{proposition:iso-tangent}, we hence obtain isomorphisms
  \begin{equation}
    \label{equation:def-isos}
    \HH^i(X,\tangent_X) \xrightarrow{\cong} \HH^i(\hilbtwo,\tangent_{\hilbtwo}),
  \end{equation}
  for~$i \geq 0$.
\end{proof}

This gives a no-go result in noncommutative algebraic geometry. As discussed in \cref{remark:noncommutative-deformations} one could hope that the Hilbert square of~$\mathbb{P}^3$ has commutative deformations, which would shed light on the noncommutative deformations of~$\mathbb{P}^3$, but the following corollary rules this out.

\begin{corollary}
  Let~$X$ be of dimension~$d\geq 3$, such that~$\mathcal{O}_X$ is exceptional. If~$X$ is infinitesimally rigid, i.e.~$\HH^1(X,\tangent_X)=0$ then so is~$\hilbtwo$.
\end{corollary}

More generally, we obtain a comparison between the deformation functor for~$X$ and that of~$\hilbtwo$ as in \cref{theorem:fantechi}. If~$\mathcal{X}\to T$ is a smooth projective family, such that~$\mathcal{X}_0\cong X$, then taking the relative Hilbert scheme produces a natural transformation of deformation functors
\begin{equation}
  \Def_X\to\Def_{\hilbtwo}.
\end{equation}
As in \cite{MR1354269}, \cref{theorem:main-comparison} should imply that this is a natural equivalence, but we have not checked the details.

\printbibliography

\emph{Pieter Belmans}, \texttt{pbelmans@mpim-bonn.mpg.de} \\
\textsc{Max Planck Institute for Mathematics, Vivatsgasse 7, 53111 Bonn, Germany}

\emph{Lie Fu}, \texttt{fu@math.univ-lyon1.fr} \\
\textsc{Universit\'e Claude Bernard Lyon 1, 43 bd du 11 novembre 1918, 69622 Ville\-urbanne cedex, France}

\emph{Theo Raedschelders}, \texttt{theo.raedschelders@glasgow.ac.uk} \\
\textsc{University of Glasgow, University Place, Glasgow, G12 8SQ, United Kingdom}

\end{document}